\def\draft{n}
\documentclass[11pt]{amsart}
\usepackage[headings]{fullpage}
\usepackage{amssymb,epic,eepic,epsfig,amsbsy,amsmath,amscd,color}
\usepackage[all]{xy}
\usepackage{graphicx}
\usepackage{texdraw}
\usepackage{url}
\usepackage{bbm}
\usepackage{mathrsfs}
\usepackage{accents}
\makeatletter
\newcommand{\sbullet}{%
  \hbox{\fontfamily{lmr}\fontsize{.4\dimexpr(\f@size pt)}{0}\selectfont\textbullet}}

\makeatother

\def\printname#1{
        \if\draft y
                \smash{\makebox[0pt]{\hspace{-0.5in}
                        \raisebox{8pt}{\tt\tiny #1}}}
        \fi
}

\def\printname#1{
        \if\draft y
                \smash{\makebox[0pt]{\hspace{-0.5in}
                        \raisebox{8pt}{\tt\tiny #1}}}
        \fi
}

\newlength{\standardunitlength}
\catcode`\@=11
\long\def\@makecaption#1#2{%
     \vskip 10pt

\setbox\@tempboxa\hbox{%\ifvoid\tinybox\else\box\tinybox\fi
       \small\sf{\bfcaptionfont #1. }\ignorespaces #2}%
     \ifdim \wd\@tempboxa >\captionwidth {%
         \rightskip=\@captionmargin\leftskip=\@captionmargin
         \unhbox\@tempboxa\par}%
       \else
         \hbox to\hsize{\hfil\box\@tempboxa\hfil}%
     \fi}
\font\bfcaptionfont=cmssbx10 scaled \magstephalf
\newdimen\@captionmargin\@captionmargin=2\parindent
\newdimen\captionwidth\captionwidth=\hsize
\catcode`\@=12

\def\lbl#1{\label{#1}\printname{#1}}

                        \theoremstyle{plain}

\newtheorem{theorem}{Theorem}[section]
\newtheorem{thm}{Theorem}
\newtheorem{lemma}[theorem]{Lemma}
\newtheorem{corollary}[theorem]{Corollary}
\newtheorem{proposition}[theorem]{Proposition}

\theoremstyle{definition}

\newtheorem{remark}[theorem]{Remark}

\def\BC{\mathbb C}
\def\BN{\mathbb N}
\def\BZ{\mathbb Z}

\def\BT{\mathbb T}
\def\BQ{\mathbb Q}

\def\CP{\mathcal P}

\def\la{\langle}
\def\ra{\rangle}

\DeclareMathOperator{\tr}{\mathrm tr}

\def\al{\alpha}
\def\ve{\varepsilon}
\def\be { \begin{equation} }
\def\ee { \end{equation} }

\def\bD{{\bar \Delta }}

\def\bs{{\widetilde {\mathfrak s}}}

\newcommand\no[1]{}

\def\hD{{\hat D}}
\def\hP{\hat \CP}

\def\hS{{\hat \CS}}

\def\bT{\mathbb T}

                  \def\uu{\mathfrak u}
      \def\nc{\newcommand}

                  \nc\FI[2]{\begin{figure}
    \begin{center}\input{#1.pstex_t}\end{center}
    \caption{#2}
    \lbl{#1}
  \end{figure}}
\nc\FIG[3]{\begin{figure}
    \includegraphics[#3]{#1.eps}
    \caption{#2}
    \lbl{fig:#1}
    \end{figure}}
\nc\FF[3]{\begin{figure}
    \includegraphics[#3]{#1.eps}
    \caption{#2}
    \lbl{#1}
    \end{figure}}
    \nc\FIGc[3]{\begin{figure}[htpb]
    \includegraphics[height=#3]{#1.eps}
    \caption{#2}
    \lbl{fig:#1}
    \end{figure}}

    \nc\FIGh[3]{\begin{figure}[htpb]
    \includegraphics[height=#3]{draws/#1.eps}
    \caption{#2}
    \lbl{fig:#1}
    \end{figure}}

\newcommand*{\lembed}{\ensuremath{\lhook\joinrel\relbar\joinrel\rightarrow}}
\newcommand*{\rembed}{\ensuremath{\leftarrow\joinrel\relbar\joinrel\rhook}}
\def\id{{\mathrm{id}}}
\def\tY{\tilde{\cY}}

\def\rk{\mathrm{rk}}

\def\cS{\mathscr S}
\def\ot{\otimes}
\def\oS{\overset \circ {\cS}}

\def\tT{\tilde \bT}
\def\cE{\mathcal E}
\def\cF{\mathcal F}

\def\Mat{\mathrm{Mat}}
\def\bk{\mathbf k}
\def\bn{\mathbf n}
\def\oD{{\mathring\Delta}}

\def\oQ{{\mathring Q}}
\def\cP{\mathcal P}

\def\bV{\mathbb V}

\def\Id{\mathrm{Id}}
\def\fS{\mathfrak S}

\def\D{\Delta}

\def\oL{\overset \circ \Lambda}
\def\QQ{K}
\def\cY{\mathcal Y}
\def\ev{{\mathrm{bl}}}

\def\embed{\hookrightarrow}

\def\sX{\mathfrak X}

\def\oL{\overset \circ \Lambda}

\def\cT{\mathcal T}

\def\cY{\mathcal Y}
\def\bQ{{\overline Q}}
\def\ev{{\mathrm{bl}}}

\def\ww{w}
\def\tw{\tilde{\ww}}
\def\Col{\mathrm{Col}}

\def\oL{\overset \circ \Lambda}
\def\oA{\overset \circ A}
\def\embed{\hookrightarrow}

\def\vp{\varphi}
\def\vk{\varkappa}

\def\beq{\overset \bullet =}

\def\oS{\mathring \Sigma}
\def\SM{(\Sigma,\cP)}
\def\pS{\partial \Sigma}

\def\ooS{\mathring {\cS}}
\def\rprod{\operatorname{  \overrightarrow{\prod}  }  }
\def\com{\mathcal C}
\def\cA{{\mathcal A}}
\def\cR{{\mathcal R}}

\def\sS{{\mathfrak S}}

\def\bsS{\bar{\sS}}

\def\La{\Lambda}
\def\oL{\mathring \Lambda}

\def\fM{\mathfrak M}
\def\cI{\mathcal I}

\def\bPsi{\overline{\psi}}

\def\bvp{\bar \varphi}

\def\bX{\bar{\sX}}
\def\bXD{\bX(\D)}
\def\bvpD{{\bar{\vp}_\D}}

\def\bvk{\bar\varkappa}
\def\hYeD{\tilde \cY^\ev(\D)}
\def\hYeDp{\tilde \cY^\ev(\D')}
\def\tYtD{\tY^{(2)}(\D)}
\def\hXD{\tilde\sX(\D)}
\def\hXDp{\tilde\sX(\D')}
\def\YeD{\cY^\ev(\D)}

\def\cZ{\mathcal Z}

\def\ZeL{\cY^\ev(\La)}

\def\Rp{R_\partial}
\def\XD{\sX(\D)}
\def\XppD{\sX_{++}(\D)}
\def\vpD{\vp_\D}

\def\XhalfD{\sX^{(\frac 12)}(\D)}
\def\XD{\sX(\D)}
\def\YD{\cY(\D)}
\def\tYD{\tilde \cY(\D)}
\def\YtD{\cY^{(2)}(\D)}
\def\tYtD{\tilde \cY^{(2)}(\D)}
\def\bve{{\boldsymbol{\ve}}}
\def\bm{{\mathbf m}}
\def\hYeD{\tilde \cY^\ev(\D)}
\def\YeD{\cY^\ev(\D)}

\def\XhD{\sX^{(\frac 12)}(\D)}

\def\ttS{\tilde{\cS}}
\def\La{{\Lambda}}

\def\tvarphi{\tilde \varphi}
\def\tXD{\tilde \sX(\D)}
\def\tpsi{\tilde \psi}
\begin{document}

\title{Quantum Teichm\"uller spaces and quantum trace map}

\author[Thang  T. Q. L\^e]{Thang  T. Q. L\^e}
\address{School of Mathematics, 686 Cherry Street,
 Georgia Tech, Atlanta, GA 30332, USA}
\email{letu@math.gatech.edu}
\dedicatory{Dedicated to Francis Bonahon on the occasion
     of his $60$th birthday}
%\date{\today}

\thanks{Supported in part by National Science Foundation. \\
2010 {\em Mathematics Classification:} Primary 57N10. Secondary 57M25.}

\begin{abstract}
We show how the quantum trace map of Bonahon and Wong can be constructed in a natural way using the skein algebra of Muller, which is an extension of the Kauffman bracket skein algebra of surfaces. We also show that the quantum Teichm\"uller space of a marked surface, defined by Chekhov-Fock (and Kashaev) in an abstract way, can be realized as a concrete subalgebra of the skew field of the skein algebra.
\end{abstract}

\maketitle

\section{Introduction}

\subsection{Quantum trace map for triangulated marked surfaces} Suppose $\SM$ is {\em marked surface}, i.e. $\Sigma$ is an oriented connected compact surface with boundary $\pS$ and $\cP\subset \pS$ is  finite set of marked points.
 F. Bonahon and H. Wong \cite{BW0} constructed a remarkable injective algebra homomorphism, called the {\em quantum trace map}
\be
\lbl{eq.1m}
 \tr_q^\D:\ooS \to \YeD,
 \ee
where $\ooS$ is the Kauffman bracket skein algebra of $\Sigma$ and $\YeD$ is the square root version of the Chekhov-Fock algebra of $\SM$. We recall the definitions of $\ooS$ and $\YeD$ in Section \ref{sec:markedsur}. While the skein algebra $\ooS$ does not depend on $\cP$ nor any triangulation, the square root Chekhov-Fock algebra $\YeD$ depends on a  $\cP$-triangulation $\D$ of $\Sigma$, i.e. a triangulation whose set of vertices is~$\cP$.

The skein algebra $\ooS$ was introduced by J. Przytycki \cite{Przy} and V. Turaev \cite{Turaev} based on the Kauffman bracket \cite{Kauffman}, and is  a quantization of the $SL_2$-character variety of $\Sigma$ along the Goldman-Weil-Petersson Poisson form, see \cite{BFK,Bullock,PS,Turaev}. The Chekhov Fock algebra in our paper is  the multiplicative version  of the one originally defined by Chekhov and Fock \cite{CF}. The theory of this multiplicative version and its square root version was developed by Bonahon, Liu, and Hiatt  \cite{BL,Liu,Hiatt}.
The Chekhov-Fock algebra is a quantization of the enhanced Teichm\"uller space using Thurston's shear coordinates, also along the Goldman-Weil-Petersson Poisson form. A slightly different form of a quantization of the enhanced Teichm\"uller space using shear coordinates is also introduced by Kashaev~\cite{Kashaev}.

Based on the close relation between the $SL_2$-character variety and the Teichm\"muller space, V.~Fock \cite{Fock} and Chekhov and Fock \cite{CF2} made a conjecture that a quantum trace map as in Equation \eqref{eq.1m} exists. The conjecture was proved in special cases in \cite{CP,Hiatt} and in full generality by Bonahon and Wong \cite{BW0}. When the quantum parameter $q$ is set to 1, the quantum trace map expresses the $SL_2$-trace of a loop in terms of the shear coordinates.

Skein algebras of surfaces, or more generally skein modules of 3-manifolds, are objects which can be defined using simple geometric notions but are hard to deal with since their algebra is difficult to handle. Their geometric definition helps to relate skein algebras/modules to topological objects like the fundamental groups, the Jones polynomial, etc.  For example, understanding the skein modules of knot complements can help to prove the AJ conjecture, which relates the Jones polynomial and the fundamental group of a knot \cite{Le:Apoly,LT,LZ}, and  the skein modules are used in the construction of topological quantum field theories \cite{BHMV}.
The introduction of the quantum trace map is a breakthrough in the study of skein algebras; it embeds the skein algebra $\ooS$  into quantum tori which have simple algebraic structure. For example,  representations of $\ooS$ are studied via the quantum trace map in  \cite{BW2}.
We will use quantum trace maps to the study of skein modules of knot complements in future work.
\subsection{Quantum trace map through the Muller algebra in skein theory}
The original construction of the quantum trace map in \cite{BW0} involves difficult calculations, with miraculous identities.
One of the goals of this paper is to offer another approach to  the quantum trace map of Bonahon and Wong using Muller's extension of skein algebras.
By extending the definition of skein algebras to the class of marked surfaces \cite{Muller}, we have a natural embedding of the skein algebra $\ooS$ of $\Sigma$ into the skein algebra $\cS$ of the marked surface $\SM$.
The latter, in the presence of a $\cP$-triangulation $\D$, naturally contains the positive part $\XppD$ of a nice algebra $\XD$, called the Muller algebra, which is a quantum torus  (see Section \ref{sec:markedsur} for details). Muller showed that the inclusion $\XppD\subset \cS$ leads to a natural embedding $\varphi_\D: \ooS \embed \XD$. And we want to argue that $\varphi_\D: \ooS \embed \XD$ is the same as  the quantum trace map of Bonahon and Wong, via the quantum shear-to-skein map as follows.

The Muller algebra $\XD$  is a quantum torus, constructed based on the {\em vertex matrix} of $\SM$ (see Sections \ref{sec:qtorus} and \ref{sec:markedsur}). The algebra  $\YeD$ is a subalgebra of another  quantum torus based on the {\em face matrix} of $\SM$. Using a duality between the vertex matrix and the face matrix, we construct an embedding $\psi: \YeD \embed
\XD$. Now we have two embeddings into $\XD$:
\be
\lbl{eq.1j}
\ooS \overset{\varphi_\D} \lembed \XD \overset{\psi} \rembed \YeD.
\ee

\begin{thm} \lbl{thm.1}

In Diagram \eqref{eq.1j}, the image of $\psi$  contains the image of $\varphi_\D$. The injective algebra homomorphism  $\varkappa_\D:\ooS \to \YeD$ defined by
$\varkappa_\D:= \psi^{-1} \circ \varphi_\D$ is equal to the quantum trace map of Bonahon and Wong.
\end{thm}

The map $\psi$ could be considered as a kind of Fourier transform, as it relates two types of coordinates based to matrices which are almost dual to each other. When the quantum parameter is set to 1, $\psi$ becomes the map expressing the shear coordinates in terms of Penner coordinates in the decorated Teichm\"uller space \cite{Penner}. One can show that  the Muller algebra is the exponential version of the Moyal quantization of the decorated Teichm\"uller space of the marked surface with respect to a natural linear Poisson structure.
Theorem \ref{thm.1} is proved in Section \ref{sec:shear}.

\subsection{The quantum Teichm\"uller space}   To each triangulation $\Delta$ of a marked surface $\SM$,
Chekhov and Fock defined an algebra, denoted by $\cY^{(2)}(\D)$ in this paper,  which is a subalgebra of the square root version $\YeD$ (see Section \ref{sec:shear}). To define an object not depending on triangulations, Chekhov and Fock suggested the following approach.

Being a quantum torus,  $\cY^{(2)}(\D)$ is a two-sided Ore domain and hence has a skew field $\tY^{(2)}(\D)$ (see Section \ref{sec:qtorus}). It was proved \cite{CF,Liu} that for any two triangulations $\D,\D'$ there is a natural {\em change of coordinate isomorphism} $\Theta_{\D\D'}: \tY^{(2)}(\D') \to \tY^{(2)}(\D)$.
Naturality means $\Theta_{\D\D}=\id$ and $\Theta_{\D\D''}= \Theta_{\D\D'} \circ \Theta_{\D'\D''}$ for any 3 triangulations $\D,\D',\D''$. Then one defines $\cT= \sqcup_{\D} \tY^{(2)}(\D)/\sim$, where $\sim$ is the equivalence relation defined by $a\sim b$, where $a \in \tY^{(2)}(\D)$ and $b\in \tY^{(2)}(\D') $, if $a= \Theta_{\D\D'}(b)$. The algebra $\cT$ is called the {\em quantum Teichm\"uller space of $\SM$}. This approach defines $\cT$ in an abstract way. %See \cite{Kashaev} for a slightly different version of the quantum Teichm\"uller space.
%The main idea is to quantize the enhanced Teichm\"uller space using Thurston's shear coordinates.

\def\ttSt{\ttS^{(2)}}
Using the skein algebra $\cS$ of $\SM$, we are able to realize $\cT$ as a concrete subspace of the skew field $\ttS$ of $\cS$. First, Muller \cite{Muller} shows that the embedding $\varphi_\D: \cS \embed \XD$ extends to an isomorphism of skew fields $\tvarphi_\D: \ttS \overset\cong \longrightarrow  \tXD$. Besides, the embedding $\psi: \YtD \embed
\XD$ extends to $\tilde \psi: \tYtD \embed
\tXD$.
This leads to an embedding
$$ \tpsi_\D:=(\tvarphi_\D)^{-1} \circ \tilde \psi: \tYtD\embed \ttS.$$
\begin{thm} \lbl{thm.2}
The image $ \ttSt:= \tpsi_\D(\tY^{(2)}(\D))$ in $\ttS$ does not depend on the triangulation $\D$, and the coordinate change map $\Theta_{\D\D'}$ is equal to $(\tpsi_\D)^{-1} \circ \tpsi_{\D'}$. Here $(\tpsi_\D)^{-1}$ is defined on $\ttSt$.
\end{thm}
  Thus,  $ \ttSt$ is a concrete realization of the quantum Teichm\"uller space $\cT$, not depending on any triangulation. We also give an intrinsic characterization of $\ttS^{(2)}$ using $\cP$-quadrilaterals, see Section \ref{sec:shear}.  From this point of view, the construction of the coordinate change isomorphism is natural. Theorem   \ref{thm.2} is part of Theorem~\ref{r.shearchange}, which contains also similar statements for the square root version $\YeD$.

\def\bpsi{\bar\psi}
\def\bvkL{\bar{\varkappa}_\La}
\def\YeL{\cY^\ev(\La)}
\subsection{Punctured surfaces and more general surfaces} Suppose $\sS$ is a punctured surface which is obtained from a closed oriented connected surface $\bsS$ by removing a finite set $\cP$. The skein algebra $\ooS$ of $\sS$ and the square root Chekhov-Fock algebra $\YeL$ (depending on an ideal triangulation $\La$ of $\sS$)  are defined as usual. Bonahon and Wong also showed that the quantum trace map (an injective algebra homomorphism)
$$
 \tr_q^\La:\ooS \to \YeL,
 $$
 exists in this case. However, since $\sS$ is not a marked surface in the sense of \cite{Muller}, the Muller algebra cannot be defined in this case.

To remedy this, we introduce a marked surface $\SM$ associated to $\sS$ as follows. For each $p\in \cP$ let $D_p\subset \bsS$ be a small disk such that $p\in \partial D_p$. Removing the interior of each disk $D_p$ from $\bsS$ we get $\Sigma$, which together with $\cP$ forms a marked surface $\SM$. The skein algebra of $\Sigma$ and that of $\sS$ are naturally  identified and denoted by $\ooS$.
Every $\cP$-triangulation $\La$ of $\bsS$ can be extended to a $\cP$-triangulation $\D$ of $\Sigma$. The Muller  algebra $\XD$ of $\SM$ contains a subalgebra $\bXD$ which is built by the standard generators of $\XD$ excluding the boundary elements. There is a natural projection $\pi: \XD \to \bXD$, see Section \ref{sec:puncture}. Define $\bvpD= \pi\circ \varphi_\D: \ooS\to \bXD$. The shear-to-skein map $\bar \psi: \cY^\ev(\La) \embed \bXD $ can be defined using the map $\psi: \YeD \embed \XD $. We will show that the quantum trace map of Bonahon and Wong is equal to $\bvpD: \ooS\to \bXD$, via the shear-to-skein map~$\bpsi$.
\begin{thm} In the diagram
\lbl{thm.3}
$$
\ooS \overset{\bvpD} \longrightarrow \bXD \overset{\bpsi} \rembed \YeL
$$
 the image of $\bpsi$  contains the image of $\bvpD$. The  algebra homomorphism  $\bvkL:\ooS \to \YeD$ defined by
$\bvkL:= \bpsi^{-1} \circ \bvpD$ is equal to the quantum trace map of Bonahon and Wong.
\end{thm}

 Theorem \ref{thm.3} is a special case of Theorem \ref{r.main1}, which treats more general type of punctured surfaces.

\subsection{Organization of the paper}Section \ref{sec:qtorus} presents the basics of quantum tori, including multiplicative homomorphisms which help to define the shear-to-skein maps later. In Section \ref{sec:marked3} we introduce the notion of skein modules of a marked 3-manifolds. Section \ref{sec:markedsurface} discusses the basics of  marked surfaces, including the duality between the face and the vertex matrix. In Section \ref{sec:markedsur} we  calculate the image of simple knot under $\varphi_\D$, a crucial technical step. In Sections \ref{sec:shear} and \ref{sec:puncture}  we prove the main results, while in Section \ref{sec:generators} the quantum trace of a class of simple knots is calculated.
In the Appendix we prove Theorem \ref{r.shearchange}.

\subsection{Acknowledgements} The author would like to thank  F.~Bonahon, C.~Frohman, A.~Kricker, G.~Masbaum, G.~Muller, J.~Paprocki,  A.~Sikora, and D.~Thurston for helpful discussions.  The author would also like to thank Centre for Quantum Geometry of Moduli Spaces (Arhus) and University of Zurich, where part of this work was done, for their support and hospitality.

\section{Quantum torus}
\lbl{sec:qtorus}
In this paper  $\BN,\BZ,\BQ$ are respectively the set of non-negative integers, the set of integers, and the set of rational numbers. Besides, $q^{1/8}$ is a formal parameter and
$\cR= \BZ[q^{\pm 1/8}]$.

% Let $\omega= t^{1/4}$.
\subsection{Non-commutative product and Weyl normalization} Suppose $\cA$ is an $\cR$-algebra, not necessarily commutative.
% For  $x_1,\dots,x_n\in \cA$ define $$ \rprod_{i \in [1,n]} x_i = x_1 x_2 \dots x_n.$$
Two element $x,y\in \cA$ are said to be {\em $q$-commuting} if there is $\com(x,y)\in \BQ$ such that
$xy = q^{\com(x,y)} yx$. Suppose $x_1,x_2,\dots,x_n \in \cA$ are pairwise $q$-commuting with $\com(x_i,x_j) \in \frac 14 \BZ$, the {\em Weyl normalization} of $\prod_i x_i$ is defined by
$$ \left[ x_1 x_2 \dots x_n \right] := q^{-\frac 12 \sum_{i<j} \com(x_i, x_j)} x_1 x_2 \dots x_n.$$

It is known that the normalized product does not depend on the order, i.e. if $(y_1,y_2,\dots, y_n)$ is a permutation of $(x_1, x_2, \dots, x_n)$, then $[y_1y_2 \dots y_n] = [x_1 x_2 \dots x_n]$.

\subsection{Quantum torus} Let $I,J$ be finite sets.
Denote by $\Mat(I \times J, \BZ)$ the set of all $I\times J$ matrices with entries in $\BZ$, i.e.
 $A \in \Mat(I \times J, \BZ)$ is a function $A: I \times J \to \BZ$.
We write $A_{ij}$ for $A(i,j)$.

We say $A \in \Mat(I \times I, \BZ)$ is {\em antisymmetric} if $A_{ij}= - A_{ji}$.
Assume $A \in \Mat(I \times I, \BZ)$ is antisymmetric and $u=q^{m/4}$ for some $m \in \BZ$. Let $u^{1/2}= q^{m/8}$. Define
 the {\em quantum torus over $\cR$ associated to $(A,u)$} by
\begin{align*}
%\bT_+(A,u):= R\la x_i , i\in I\ra /(x_i x_j = u^{A_{ij}} x_j x_i) \\
\bT(A,u,x):= R\la x_i^{\pm 1} , i\in I\ra /(x_i x_j = u^{A_{ij}} x_j x_i).
\end{align*}
We call $x_i, i\in I$ the basis variables of the quantum torus $\bT(A,u,x)$. Letter $x$ indicates that the
basis variables are $x_i, i\in I$. We often write $\bT(A,u)= \bT(A,u,x)$ when the basis variables are
fixed.

It is known that $\bT(A,u)$ is a two-sided Noetherian domain, and hence a two-sided Ore domain, see e.g. \cite{GW}. Denote by $\tT(A,u)$ the skew field (or division algebra) of $\bT(A,u)$.

Let $\BZ^I$ be the set of all maps $\bk: I \to \BZ$. For $\bk \in \BZ^I$ define the {\em normalized monomial} $x^\bk$  by
$$ x^\bk = \left[ \prod_{i \in I} x_i^{\bk(i)} \right].$$
The set $\{ x^\bk \mid \bk \in \BZ^I\}$ is an $\cR$-basis of  $\bT(A,u,x)$.

We will consider $\bk\in \BZ^I$ as a row vector, i.e.
a matrix of size $1 \times I$. Let $\bk^\dag$ be the transpose of $\bk$.
Define an anti-symmetric $\BZ$-bilinear form on $\BZ^I$ by
$$\la \bk,\bn\ra_A:= \sum A_{ij} \bk(i) \bn(j) = \bk   A   \bn^\dag.$$

The following well-known fact follows easily from the definition.

\begin{proposition} \lbl{r.11s}
For $\bk,\bn,\bk_1,\dots,\bk_m \in \BZ^I$, one has
\begin{align}
\lbl{eq.01a}
x^{\bk} x^{\bn}  & = u^{\la \bk, \bn\ra_A} \, x^{\bn}
x^{\bk}\\
\lbl{eq.01}
x^{\bk_1} x^{\bk_2} \dots x^{\bk_m} &= u^{\frac 12 \sum_{j<l }\la \bk_j, \bk_l\ra_A} \,  x^{\sum_j\bk_j}.
\end{align}
In particular, for $n \in \BZ$ and $ \bk \in \BZ^I$, one has
$
(x^{\bk})^n = x^{ n \bk}$.
\end{proposition}
\begin{remark} \lbl{r.def7}
The quantum torus $\bT(A,u,x)$ can be defined as the free $\cR$-module with basis $\{ x^\bk \mid \bk \in\BZ^I\}$ subject to the relation \eqref{eq.01a}.
\end{remark}

\subsection{Reflection symmetry}
There is a unique $\BZ$-algebra anti-homomorphism
$$ \chi: \bT(A,u,x) \to \bT(A,u,x)$$
  satisfying
$$ \chi(q^{1/8})= q^{-1/8}, \quad \chi(x_i) = x_i  \ \forall i\in I.$$
Here $\chi$ is an algebra anti-homomorphism means $\chi(xy)= \chi(y) \chi(x)$ for all $x, y \in \bT(A,u)$.
Note that $\chi$ is an anti-involution of $\bT(A,u)$ since $\chi^2=\id$. We call $\chi$ the {\em reflection
symmetry}.
It is clear that $\chi$ extends to an anti-involution of $\tT(A,u)$.

An element $z\in \bT(A,u)$ is called {\em reflection invariant} if $\chi(z)=z$. Similarly, if
$A \in \Mat(I \times I,\BZ)$ and $B \in \Mat(J \times J,\BZ)$ are antisymmetric matrices, an $\cR$-algebra
homomorphism
$f : \bT(A,u) \to \bT(B,v)$ is said to be {\em reflection invariant}  if $f \chi = \chi f$.

From the definition, one sees that
each normalized monomial $x^\bk$ is reflection invariant.
The following simple fact will be helpful and used many times.
\begin{lemma}
\lbl{r.reflection} Suppose in  $z\in \bT(A,u)$ is reflection invariant and
\be
\lbl{eq.92}
 z = \sum_{j=1}^m q^{r_j} x^{\bk_j},
 \ee
where $r_j\in \BQ$, and $\bk_j$ are pairwise distinct.  Then all $r_j=0$, i.e.
$ z = \sum_{j=1}^m x^{\bk_j}.$
\end{lemma}
\begin{proof}
Applying $\chi$ to \eqref{eq.92}, we have
$
z = \sum_{j=1}^m q^{-r_j} x^{\bk_j}.
$
 Since $k_j$ are pairwise distinct, the presentation of $z$ as a linear combination of $x^{\bk_j}$ is
 unique. Hence we must have $q^{r_j}= q^{-r_j}$, or $r_j=0$.
\end{proof}

\begin{remark} Lemma \ref{r.reflection}
is one reason why we use $q$ as an indeterminate, not a complex number.
\end{remark}

\nc{\cB}{\mathcal B}
\subsection{Based modules}\lbl{sec.based}
A {\em based $\cR$-module} $(V,\cB)$ consists of a free $\cR$-module $V$ and a preferred base $\cB$. Another
based module $(V', \cB')$ is a {\em based submodule} of  $(V,\cB)$ if $V'\subset V$ and $\cB'\subset \cB$.
In that case,  the {\em canonical projection} $\pi: V \to V'$ is the $\cR$-linear map given by $\pi(v)=v$ if
$v\in \cB'$ and $\pi(v)=0$ if $v \in \cB\setminus \cB'$.
The following is obvious and will be useful.
\begin{lemma}
\lbl{r.basesub}
Let  $(V',\cB')$ be a based submodule of $(V,\cB)$. Suppose $a\in V'$ and
$ a= \sum_{j=1}^m q^{r_j} b_j,$
where each $b_j \in \cB$ and $r_j \in \BQ$. Then $b_j \in \cB'\subset V'$ for every $j=1,\dots,m$.
\end{lemma}

For an anti-symmetric matrix $A\in \Mat(I \times I, \BZ)$,
we will consider the quantum torus $\bT(A,u,x)$ as a based module with the preferred base $\{x^{\bk} \mid
\bk \in \BZ^I\}$.
Suppose $I'\subset I$ and $A'$ is the $I'\times I'$ submatrix of $A$. Then $\bT(A',u)$ is a based
submodule of $\bT(A,u)$. The canonical projection is not an algebra homomorphism unless $A'=A$. However,
if $V$ is the $\cR$-submodule of $\bT(A,u)$ spanned by $x^\bk$ such that $\bk(i) \ge 0 \ \forall i \in
I\setminus I'$, then the restriction of $\pi$ onto $V$ is an algebra homomorphism.

\subsection{Multiplicatively linear homomorphism} \lbl{sec.homo}
Suppose $A \in \Mat(I \times I,\BZ)$ and $B \in \Mat(J \times J,\BZ)$ are antisymmetric matrices, and both
$u,u^r$ are integral powers of $q^{1/4}$, for some rational number $r$. Consider the quantum tori
$\bT(A,u^r,x)$ and $\bT(B,u,y)$.

For a matrix $H\in \Mat(I \times J, \BZ)$, define an $\cR$-linear map
$$ \psi=\psi_H: \bT(A,u^r,x)\to \bT(B,u,y), \quad \text{by } \  \psi(x^\bk) := y^{\bk H}.$$

Denote by $H^\dag$ the transpose of $H$.

\begin{proposition} \lbl{r.homo}
(a) The above defined $\psi$ is a $\cR$-algebra homomorphism if and only if
\be
\lbl{eq.1a}
HBH^\dagger = rA.
\ee

(b) The map $\psi$ is reflection invariant.

(c) Suppose $\rk(H) = |I|$. Then  $\psi$ is injective.

\end{proposition}
\begin{proof}
(a) follows right away from \eqref{eq.01a}. See Remark \ref{r.def7}.

(b) Since  $x^\bk$ and $ y^{\bk H}$ are reflection invariant, $\psi$ is reflection invariant.

(c)  Since  $\rk(H)=|I|$,  $\psi$ maps
injectively the preferred base of $\bT(A, u^r)$ into the preferred base of $\bT(B,u)$. Hence, $\psi$ is
injective.
\end{proof}
In case $\rk(H)=|I|$,  a left inverse of $\psi$ can be given by a multiplicative linear homomorphism.

\newcommand{\qbinom}[2]{\left[\begin{matrix}
#1 \\ #2
\end{matrix}  \right]}

%%%%%%%%%%%%%%%%%%%%%%%%%%%%
%%%%%%%%%%%%%%%%%%%%%%%%%%%%

\def\np{\newpage}
\def\pM{\partial M}
\def\cN{\mathcal N}
\def\MN{(M,\cN)}
\def\tiX{\tilde {\sX}}
\def\tcS{\tilde \cS}
%\np
\section{Skein modules of 3-manifolds} \lbl{sec:marked3}
\subsection{Marked 3-manifold} {\em A marked 3-manifold $\MN$} consists of
an oriented connected 3-manifold $M$  with (possibly empty) boundary $\pM$ and  a 1-dimensional oriented
submanifold $\cN \subset \pM$ such that $\cN$ is the disjoint union of several open intervals. Here an
open interval in $\pM$ is an oriented 1-dimensional submanifold of $\pM$ diffeomorphic to the interval
$(0,1)$.

{\em An  $\cN$-link $L$} (in $M$) is a compact 1-dimensional non-oriented smooth submanifold of $M$ equipped
with a normal vector field
such that $L \cap
\cN = \partial L$  and at a boundary point in $\partial L=L \cap
\cN $, the normal vector is tangent of $\cN$ and determines the orientation of $\cN$. Here a {\em normal
vector field} on $L$  is a vector field not tangent to $L$ at any point.
The empty set is also considered an $\cN$-link.
 Two  $\cN$-links are {\em $\cN$-isotopic} if they are
 isotopic through the class of $\cN$-links. Very often we identify an $\cN$-link with its $\cN$-isotopy class.
The normal vector field is usually called a framing of $L$.
    All links considered in this paper are framed.
 \subsection{Kauffman bracket skein modules}
 Recall that $\cR= \BZ[q^{\pm 1/8}]$.
The {\em Kauffman bracket skein module}  $\cS\MN$ is  the $\cR$-module freely spanned by isotopy classes of
 $\cN$-links in $\MN$ modulo  the usual {\em skein relation} and  the {\em trivial loop relation},
and
 the new {\em trivial arc relation} (see Figure \ref{fig:skein}).
Here and in all Figures, framed links are drawn with blackboard framing.
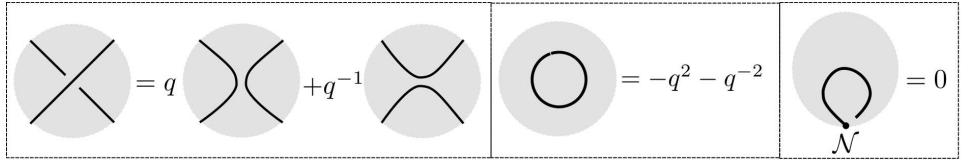
 \FIGc{skein}{Skein relation, trivial loop relation, trivial arc relation}{2.2cm}

 More precisely,
 \begin{itemize}
 \item  The skein relation: if  $L, L_+, L_-$ are  identical except in a ball in which they look like
     in Figure \ref{fig:skein1},
    then
 \be \notag
 L  = q L_+ + q^{-1} L_-
 \ee
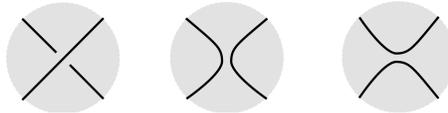
 \FIGc{skein1}{From left to right:  $L, L_+, L_-$.}{1.5cm}
\item The trivial loop relation:
 If  $L$ is a loop  bounding a disk in $M$ with framing perpendicular to the disk, then
 $$L= -q^2 -q^{-2}.$$

\item   The trivial arc relation: If $L= L' \sqcup a $, where $a$ is a trivial arc in $M \setminus L'$
    then $L=0$. Here $a$ is an trivial arc in $M \setminus L'$ means  $a$ and a part of $\cN$ co-bound
    an embedded disc in $M\setminus L'$.

\end{itemize}

\begin{proposition}\lbl{r.boundary}
In $\cS\MN$, the reordering relation depicted in Figure \ref{fig:boundary} holds.
\end{proposition}
\FIGc{boundary}{ Reordering relation.}{1.9cm}

Here in Figure \ref{fig:boundary} we assume that  $\cN$ is perpendicular to the page and  its intersection
with the page is the bullet denoted by $N$. The vector of orientation of $\cN$ is pointing to the reader.
There are two strands of the links coming to $\cN$ near $N$, the lower one being depicted by the broken
line.
\begin{proof}\FIGc{boundary.proof}{Proof of Proposition \ref{r.boundary}}{1.6cm}
The proof is given in Figure \ref{fig:boundary.proof}. Here the first identity is an isotopy, the second
is the skein relation, the third follows from the trivial arc relation.
\end{proof}

\begin{remark} (a) The orientation of $M$ is very important in the skein relation.

(b)
Kauffman bracket skein modules were introduced by J. Przytycki \cite{Przy} and V. Turaev \cite{Turaev}
for the case when the marking set $\cN$ is empty. Muller \cite{Muller} introduced Kauffman bracket skein
modules for marked surfaces, see Section \ref{sec:markedsur}.
Here we generalize Muller definition to the case of marked 3-manifolds. At first glance, our definition is
different from that of Muller for the original case of marked surfaces. The reason is Muller used
link diagrams to define the skein modules, and he has to impose the reordering relation since it is not a
consequence of other relations if one considers link diagrams. Here we use links in 3-manifolds, and the
reordering relation is a consequence of the other relations and the topology afforded by the 3-rd
dimension.
\end{remark}

\def\ocP{\mathring {\cP}}

\section{Generalized marked  Surfaces}
\lbl{sec:markedsurface}

Here we present basic facts about surfaces, their triangulations, and the vertex matrix and the
face matrix associated to a triangulation. In Subsection \ref{sec.dual} we discuss the duality between the
vertex matrix and the face matrix.

%, their skein algebras, and Muller's result on embedding of skein algebras of marked surfaces into quantum tori.

\subsection{Definitions and basic facts} {\em A generalized marked surface} $\SM$ consists of
 a connected compact oriented surface $\Sigma$ with (possibly empty) boundary $\pS$, and a finite set $\cP \subset \Sigma$. Elements of $\cP$ are called {\em marked points}.
If $\cP\subset \pS$, then $\SM$ is called a {\em marked surface}\footnote{Our generalized marked surface is the same as  ``punctured surface with boundary" in \cite{BW0}, and
our marked surfaces is the same as ``marked surface" in \cite{Muller}.}.

\def\pr{\operatorname{pr}}

 %Let $\oS= \Sigma \setminus \partial \Sigma$, which is a surface without boundary.

%Let us fix a marked surface $\SM$ in this whole section.

A {\em $\cP$-link} in $\Sigma$ is an immersion $\al: C\to \Sigma$, where $C$ is compact 1-dimensional non-oriented
manifold, such that
\begin{itemize}
\item the restriction of $\al$ onto the interior of $C$ is an embedding into  $\Sigma \setminus \cP$,
    and
\item $\al$ maps the boundary of $C$ into $\cP$.
    \end{itemize}
    The image of a connected component of $C$ is called a component of $\al$. When $C$ is a $S^1$, we call
    $\al$ an {\em $\cP$-knot}, and when $C$ is $[0,1]$, we call $\al$ a {\em $\cP$-arc}.
 \no{   In the latter case,
    the images of $[0,1/2)$ and $(1/2,1]$ under $\al$ are called  {\em half-arcs} of $\al$.
    A {\em
$\cP$-half-arc} is a half-arc of a $\CP$-arc.
   % An $\cN$-knot is simply a framed unoriented knot in $M \setminus \cN$.
   }
Two $\cP$-links are {\em $\cP$-isotopic} if they are isotopic in the class of $\cP$-links.
 Very often we identify a $\cP$-link with %its $\cP$-isotopy class and
 its image in $\Sigma$.

Suppose $\al, \beta$ are
$\cP$-links. An {\em internal common point} of $\al$ and $\beta$ is point  in $(\al\cap \beta) \setminus \cP$.
 %We say that two $\cP$-links  {\em do not intersect internally} if they don't have a common point in $\Sigma \setminus \cP$.
  Let $\mu(\al,\beta)$ denote the minimum number of internal common points of $\al'$ and $\beta'$, over all transverse pairs $(\al',\beta')$ such that $\al'$ is $\cP$-isotopic to $\al$ and $\beta'$ is $\cP$-isotopic to $\beta$.  It is known that there is a $\cP$-link $\gamma$ $\cP$-isotopic to $\beta$ such that
 $|\gamma\cap a| = \mu(\beta,a)$ for any component $a$ of $\al$, see \cite{FHS82,FST}.

A $\cP$-link is {\em essential} if it does not have a component bounding a disk whose interior is in $\Sigma\setminus \cP$; such a
component is either a smooth trivial knot in $\Sigma \setminus \cP$, or a closed $\cP$-arc bounding a disk whose interior is in
$\Sigma \setminus \cP$. By convention, the empty set is considered an essential $\cP$-link.

A $\cP$-arc is called a {\em boundary arc} if it is $\cP$-isotopic to an arc in $\pS$. A $\cP$-arc is {\em
inner} if it is not a boundary arc.

\subsection{Triangulation} \lbl{sec.42}
 {\em A $\cP$-triangulation of $\Sigma$}, also called a triangulation of $\SM$, is a  triangulation of $\Sigma$ whose set of vertices is $\cP$. We will always assume that $\SM$ is {\em triangulable}, i.e. it has at least one $\cP$-triangulation. It is known that $\SM$ is triangulable if and only if every connected component of $\pS$ has at least one marked point and $M$ is not one of the following:
 \begin{itemize}
 \item a sphere with one or two marked points;
 \item a monogon with no interior marked point; or
 \item a digon with no interior marked point;
 \end{itemize}

 For a  triangulable generalized marked surface $\SM$, one can
use the following more technical definition (see e.g. \cite{FST,Muller}) of triangulation.

{\em A $\cP$-triangulation of $\Sigma$} is a collection $\D$ of $\cP$-arcs such that
\begin{itemize}
\item[(i)] no two $\cP$-arcs in $\D$ intersect in $\Sigma \setminus \cP$ and no two are $\cP$-isotopic,
    and
\item[(ii)] $\D$ is maximal amongst all collections of $\cP$-arcs with the above property.
\end{itemize}

An element of $\D$ is called an {\em edge} of the triangulation.
It can be proved that if $\D$ is a  triangulation, then one can replace $\cP$-arcs in $\D$ by
 $\cP$-arcs in their respective $\cP$-isotopy  classes such that  every boundary arc in $\D$ does lie on the
 boundary $\pS$. We always assume the $\cP$-arcs in a  triangulation  satisfy this requirement.

{\em A triangulated generalized marked surface} is a generalized marked surface equipped with a triangulation.

A {\em $\cP$-$n$-gon} is a  smooth map $\gamma: \sigma \to \Sigma$ from a regular $n$-gon $\sigma$ (in the standard plane)  to $\Sigma$ such that
(a) the restriction of $\gamma$ onto the interior $ \mathring \sigma$ of
$\sigma$ is a diffeomorphism onto its image, (b) the restriction of $\gamma$
onto each edge of $\sigma$ is a $\cP$-arc, called an edge of $\gamma$.

A  $\cP$-triangulation  $\D$ cuts $\Sigma$ into {\em $\cP$-triangles}, i.e. the closure of each connected component of $\Sigma \setminus E_\D$, where $E_\D= \cup_{a\in \D}a$, has the structure of a
$\cP$-triangle.  Denote by $\cF(\D)$ the set of all triangles of the triangulation $\D$.
Note that two edges of a  triangle $\tau \in \cF(\D)$  either coincide  (i.e. have the same images) or  do not have internal common points and are not $\cP$-isotopic. When two edges of a  triangle $\tau \in \cF(\D)$  coincide, $\tau$ is called a {\em self-folded triangle}, see Figure \ref{fig:triangle1}.
If $\SM$ is a marked surface, i.e. $\cP\subset \pS$, then a triangulation of $\SM$ cannot have self-folded triangle.

\FIGc{triangle1}{A triangle and a self-folded triangle}{1.5cm}

A $\cP$-knot is said to be {\em $\D$-normal}, where $\D$ is a $\cP$-triangulation, if $\al$ is non-trivial and $|\al\cap a| = \mu(\al,a)$ for all $a\in \D$.
%It is easy to show that $\al$ is $\D$-normal if and only if  every connected component of $\al \cap \tau$, where $\tau$ is a triangle of $\D$, is an interval whose end points lie on different edges of $\tau$.
Every non-trivial $\cP$-knot  is $\cP$-isotopic   to a $\D$-normal knot.

\subsection{Face matrix}\lbl{sec.51} Let $\D$ be a triangulation of generalized marked surface $\SM$ and $\tau\in \cF(\D)$, i.e. $\tau$ is a triangle of $\D$.
We define a anti-symmetric matrix $Q_\tau\in \Mat(\D\times \D,\BZ)$ as follows. If $\tau$ is a self-folded triangle, then let $Q_\tau$ be the 0 matrix.
If  $\tau$ is not self-folded and hence has 3 distinct edges $a,b,c$ in counterclockwise order (see  Figure \ref{fig:triangle1}), then define
 \begin{align*}
  Q_\tau(a,b)& =Q_\tau(b,c)=Q_\tau(c,a)=1 \\
  Q_\tau(e,e')& =0 \quad \text{if  one of $e,e'$ is not in $\{a,b,c\}$ }.
 \end{align*}
In other words,    $Q_\tau\in \Mat(\D \times \D,\BZ)$ is the 0-extension of the following $\{a,b,c\}
\times \{a,b,c\}$ matrix %(with order $a,b,c$)
 \be
 \lbl{eq.Qtau}
 \begin{pmatrix}
 0 & 1 & -1 \\
 -1 &0 & 1 \\
 1 & -1 &0
 \end{pmatrix}.
 \ee

\def\oD{{\mathring {\D}}}

Define the {\em face matrix} $Q=Q_\D \in \Mat(\D \times \D, \BZ)$ by
$$ Q = \sum_{\tau\in \cF(\D)} Q_\tau.$$

\begin{lemma}\lbl{r.Qtau0}  Suppose
 $\tau\in \cF(\D)$ has edges $a,b,c$ as in Figure \ref{fig:triangle1} and $\bk\in \BZ^\D$. Then
\be
\lbl{eq.Qtau0}
(\bk Q_\tau)(c) = \bk(b) - \bk(a)
\ee
\end{lemma}
\begin{proof}
The proof follows immediately from the explicit form \eqref{eq.Qtau} of $Q_\tau$.
\end{proof}

\begin{remark}
Our $Q$ is the same as $Q^{\D}$ of \cite{Muller} or the same as $-B$ of \cite{FST}, and is also known as
the signed adjacency matrix. We use the terminology ``face matrix" to emphasize the duality with
the ``vertex matrix".
\end{remark}

\def\QQ{H}

\subsection{Marked surface and its vertex matrix} \lbl{sec.vmatrix} Assume $\SM$ is a triangulable marked surface. In particular, $\cP\subset \pS$. Let $\D$ be a triangulation of $\SM$.

For each edge $a\in \D$ choose  an interior point in $a$. Removing this interior point, from $a$ we get two {\em half-edges}, each  is incident to exactly one vertex in $\cP$.
Suppose $p\in \cP$ and $a',b'$ are two half-edges (of two different edges) incident to $p$. Define
$P_p(a',b')$ as in Figure \ref{fig:vmatrix}, i.e.
 $$P_p(a',b')= \begin{cases}
1  &\text{if $a'$ is clockwise to $b'$ (at vertex $p$)}\\
-1  &\text{if $a' $ is counter-clockwise to $b'$ (at vertex $p$)}.
\end{cases}
$$
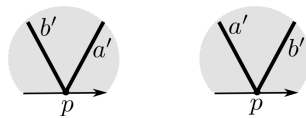
\FIGc{vmatrix}{$P_p(a',b')=1$ for the left case, and $P_p(a',b')=-1$ for the right one. Here the shaded area is part of $\Sigma$, and the arrow edge is part of a boundary edge. There might be other half-edges incident to $p$, and they maybe inside and outside the angle between $a'$ and $b'$ }{1.5cm}

Also, if one of $a', b'$  is not incident to $p$, set $P_p(a',b')=0$.
Define the {\em vertex matrix} $P=P(\D)\in \Mat(\D\times \D, \BZ)$ by
$$ P(a,b)= \sum P_p(a',b'),$$
where the sum is over all $p\in \cP$, all half-edges $a'$ of $a$, and all half-edges $b'$ of $b$.

\def\thD{\theta_\D}
\def\aa{d}
\def\Rp{R_\partial}

\begin{remark}
 The fact that $\cP\subset \pS$ is crucial for the definition of the vertex matrix.
 The vertex matrix were first introduced in \cite{Muller}, where it is  called the orientation matrix.
\end{remark}

\subsection{Vertex matrix versus face matrix}\lbl{sec.dual}
The following relation between the face and the vertex matrices of a marked surface is important for us.

Recall that an edge $a\in \D$ is a {\em boundary edge} if it is a boundary $\cP$-arc, otherwise it is called an {\em
inner edge}. Let  $\oD$ be the set of all inner edges.
Let $\oQ $ be the $(\oD  \times \oD)$-submatrix of $Q$
and $\QQ$ be the $(\oD  \times \D)$-submatrix of $Q$.
\begin{lemma} \lbl{r.51a} Suppose $\SM$ is a marked surface with a triangulation $\D$.

(a) One has $\QQ P \QQ^\dag = -4 \oQ$.

(b) The rank of $\QQ$ is $|\oD |$.

%\red{(c) The sum of entries of each column of $Q_\tau$ or $Q$ is 0.}
\end{lemma}
\begin{proof}
(a) Let $\id_{\D \times \oD } \in \Mat(\D \times \oD , \BZ)$ be the matrix which has 1 on the main
diagonal, and 0 everywhere else, i.e. $\id_{\D \times \oD }(a,b)=\delta_{a,b}$. By \cite[Proposition
7.8]{Muller},
\be
\lbl{eq.61}
P\QQ^\dag= - 4 \id_{\D \times \oD }.
\ee
Hence $\QQ P \QQ^\dag = -4 \QQ \, \id_{\D \times \oD } = -4 \oQ$.

(b) Since $\QQ$ has $|\oD  |$ rows, $\rk(\QQ) \le |\oD |$. Because $\rk(\id_{\D \times \oD }) = |\oD |$,
Equation \eqref{eq.61} shows that
$\rk(\QQ) \ge |\oD |$. Hence $\rk(\QQ) = |\oD |$.
%\red{ (c) is obvious from  \eqref{eq.Qtau}.}
\end{proof}

\section{Skein algebra of marked surfaces}
\lbl{sec:markedsur}

Throughout this section we fix a marked surface $\SM$.

\subsection{Skein module of  marked surface} Let $M$ be the cylinder over $\Sigma$ and $\cN$ the cylinder over $\cP$, i.e. $M= \Sigma \times (-1,1)$
and $\cN= \cP \times (-1,1)$.
 We consider $\MN$ as a marked 3-manifold, where the orientation on each component of $\cN$  is given by
 the natural orientation of $(-1,1)$. We  identify $\Sigma$
 with $\Sigma \times \{0\} \subset M$. There is a vertical projection $\pr:M \to \Sigma$, mapping $(z,t)$ to $z$. The number $t$ is called the {\em height} of $(z,t)$. The {\em vertical vector} at $(z,t)\in \Sigma \times (-1,1)$ is the unit vector tangent to $z\times (-1,1)$ and having direction the positive orientation of $(-1,1)$.

Define $\cS\SM:= \cS\MN$. Since we fix $\SM$, we will denote $\cS$ for $\cS\SM$ in this section.

Suppose $\al\subset  \Sigma$ is a $\cP$-link. We will define  $[\al]\in \cS$ as follows.
  Let $\al'\subset \Sigma \times (-1,1)$ be an $\cN$-link such that
  \begin{itemize}
  \item[(i)]  $\pr(\al')=\al$,  the framing of $\al'$ is vertical everywhere, and
  \item[(ii)] for every $p\in \cP$, if $a_1,\dots, a_{k_p}$ are strands of $\al$ (in a small neighborhood of $p$) coming to $p$ in clockwise order, and $a_1',\dots, a'_{k_p}$ are the strands of $\al'$ projecting correspondingly onto $a_1,\dots, a_{k_p}$, then the height of $a_i'$ is greater than that of $a'_{i+1}$ for $i=1, \dots, k_p-1$. See an example with $k_p=3$ in Figure \ref{fig:simul}.
  \end{itemize}
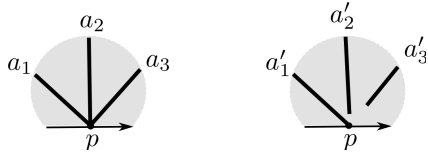
  \FIGc{simul}{Left: There are 3 strands $a_1, a_2, a_3$ of $\al$ coming to $p$, ordered clockwise. Right: The corresponding strands $a'_1, a'_2, a'_3$ of $\al'$, with $a'_1$ above $a'_2$, and $a'_2$ above $a'_3$.}{2cm}

  It is clear that the $\cN$-isotopy class of $\al'$ is determined by $\al$. Define $\al$ as an element in $\cS$ by
  \be
  \lbl{eq.simul}
   [\al] = q^{\frac 14 \sum_{p\in \cP} (k_p-1)(k_p-2)} \al' .
   \ee
 % By convention, if $\al$ is the empty link, then $[\al]= 1\in \cR

The factor which is a power of $q$ on the right hand side is introduced so that $[\al]$ is invariant under a certain transformation, see Section \ref{sec.reflection}.
By \cite[Lemma 4.1]{Muller}, we have the following fact, which had been known for unmarked surface~\cite{PS}.

\begin{proposition}[\cite{Muller}]
\lbl{r.basis} As an $\cR$-module, $\cS$ is a free  with basis
the set of all $[\al]$, where $\al$ runs the set of all $\cP$-isotopy classes of essential $\cP$-links.
\end{proposition}
A concise and simple proof of this fact can be obtained using the Diamond Lemma as in \cite{SikoraW}.
We will consider $\cS$ as a based $\cR$-module with the preferred base described by Proposition \ref{r.basis}. In what follows we often use the same notation, say $\al$, to denote a $\cP$-link and the  element $[\al]$ of $\cS$ when there is no confusion.

\subsection{Algebra structure and reflection anti-involution}\lbl{sec.reflection}
For $\cN$-links $\al_1, \al_2$ in $M=\Sigma \times (-1,1)$, considered as elements of $\cS$, define the product
 $\al_1 \al_2$ as the result of stacking $\al_1$  atop  $\al_2$ using the cylinder structure of $\MN$. Precisely
 this means the following. Let $\iota_1: \Sigma \times (-1,1) \embed \Sigma \times (-1,1)$ be the
 embedding $\iota_1(x,t)= (x, \frac{t+1}2)$ and $\iota_2: \Sigma \times (-1,1) \embed \Sigma \times
 (-1,1)$ be the embedding
 $\iota_1(x,t)= (x, \frac{t-1}2)$. Then $\al_1 \al_2:= \iota_1(\al_1) \cup \iota_2(\al_2)$.
This product makes $\cS$ an $\cR$-algebra, which is non-commutative in general.

Let $\chi: \cS \to \cS$ be the bar homomorphism of \cite{Muller}, which is  the $\BZ$-algebra
anti-homomorphism defined by (i)
$ \chi(q^{1/8})= q^{-1/8}$ and (ii) $\chi(L)$ is the reflection image of $L$ for any $\cN$-link $L$ in $\Sigma
\times (-1,1)$. Here the reflection is the map $(z,t)\to (z, -t)$ of $\Sigma \times (-1,1)$. It is clear that
$\chi$ is an anti-involution. An element $\al\in \cS\SM$ is {\em reflection invariant} if $\chi(\al)=\al$. From the reordering relation (Proposition \ref{r.boundary}) one can easily show that $[\al]$ is reflection-invariant for any $\cP$-link $\al$.

  \subsection{Functoriality} Let $(\Sigma',\cP')$ be a marked surface such that $\Sigma'\subset \Sigma$
  and $\cP'\subset \cP$. The embedding $\iota:\Sigma'\embed \Sigma$ induces an $\cR$-algebra homomorphism
  $\iota_*:\cS(\Sigma', \cP')\to \cS\SM$.

   If $\Sigma' = \Sigma$, then $\iota_*:\cS(\Sigma, \cP')\to \cS\SM$ is injective, because the preferred
   basis of $\cS(\Sigma, \cP')$ is a subset of that of $\cS\SM$.

  In particular, the natural $\cR$-algebra homomorphism $\iota_*:\ooS \to \cS$ is injective, where $\ooS=
  \cS(\Sigma, \emptyset)$. We will always identify $\ooS$ with a subset of $\cS$ via $\iota_*$.

\def\Rp{\cR}
\subsection{Muller's algebra: quantum torus associated to vertex matrix} Suppose $\SM$ has a triangulation
$\D$.
 By definition, each $a\in \Delta$ is an $\cN$-arc (with vertical framing), and we consider $a$ as an
 element of the skein algebra $\cS$. From the reordering relation (Proposition \ref{r.boundary}) we see
 that for each pair $a,b\in \D$ one has
\be
\lbl{eq.35}
ab = q^{P(a,b)} ba,
\ee
where $P\in \Mat(\D\times \D,\BZ)$ is the vertex matrix (see Subsection \ref{sec.vmatrix}).
It is the $q$-commutativity of edges of $\D$, equation \eqref{eq.35},  that leads to the introduction of the vertex matrix in \cite{Muller}.

The {\em Muller algebra} $\sX(\D)$ is defined to be the quantum torus $\bT(P,q,X)$, i.e.
$$ \sX(\D) = \cR\la X_a^{\pm 1}, a \in \D \ra /(X_a X_b = q^{P(a,b)} X_b X_a).$$
Denote by $\tiX(\D)$ the skew field of $\sX(\D)$.
Recall that $\XD$ is a based $\cR$-module with preferred basis $\{ X^\bk \mid \bk \in \BZ^\D\}$. Let
$\XppD$ be the $\cR$-submodule of $\XD$ spanned by $X^\bk$ with $\bk\in \BN^\D$, i.e. $\bk(a) \ge 0$ for
all $a\in \D$. Then $\XppD$ is an $\cR$-subalgebra of $\XD$.

\def\thD{\phi_\D}
\def\fM{\mathfrak M}
Relation \eqref{eq.35} shows that there is a unique algebra homomorphism
\be
\lbl{eq.thD}
\thD: \XppD \to \cS, \quad \text{defined by $\thD(a)= X_a$}.
\ee
 For $\bk\in \BN^\D$, the image $\D^\bk:=\thD(X^\bk)$ has a transparent geometric description. In fact, as
 observed in \cite{Muller},
 $\D^\bk$ is a $\cP$-link consisting of $\bk(a)$ copies of $a$ for every $a\in \D$. Here each copy of $a$,
 by definition, is a $\cP$-arc $\cP$-isotopic   to $a$ in $\Sigma \setminus \left( \cup_{b \in \D
 \setminus\{a\}}b \right)$.
This gives a nice geometric interpretation of the Weyl normalization.

The following is one of the main results of \cite{Muller}.
\begin{theorem}[Muller]
\lbl{r.Muller}

(i) The homomorphism $\thD$ in \eqref{eq.thD} is injective.

(ii) There is a unique injective algebra homomorphism $\vpD:\cS \embed \sX(\D)$ such that $\vpD \circ
\thD$ is the identity on $\XppD$. In other words, the combination
\be
\lbl{eq.incl1}
  \XppD \overset{\thD} \lembed \cS \overset{\vpD} \lembed \XD
  \ee
is the natural embedding $\XppD \embed \XD$.
Besides, $\vpD$ is reflection invariant, i.e. $\vpD$ commutes with $\chi$.
\end{theorem}
We will call $\vpD$ the {\em skein coordinate map of $\cS$ associated to the triangulation} $\D$.
The skein coordinates,  in the classical case $q=1$, correspond to the
 Penner coordinates on decorated Teichm\"uller spaces  \cite{Penner}.
We will identify $\cS$ with a subset of $\XD$ via the embedding $\vpD$. Then $\cS$ is sandwiched between
$\XppD$ and $\XD$.

While  $\cS$ is a complicated algebra, $\XD$ has a simple algebraic structure. Being a subring of a
two-sided Noetherian domain $\XD$, $\cS$ is a two-sided Noetherian domain, and hence a two-sided Ore
domain, see \cite{GW}. It follows that the skew field $\tilde \cS$  of $\cS$ exists.
The inclusions in~\eqref{eq.incl1} shows that $\tiX(\D) = \tilde \cS$.
The inclusions in \eqref{eq.incl1} also show that $\cS$ is an essential subalgebra of $\XD$ in the sense
that every algebra homomophism from $\XD$ to another algebra is totally determined by it restriction on
$\cS$.

\def\XhD{\sX^{(\frac 12)}(\D)}
 Let $\XhD$ be the quantum torus $\bT(P,q^{1/4},x)$, which has
basis variables $x_a, a \in \D$. We consider $\XD$ as a $\Rp$-subalgebra of $\XhD$ by setting
$X_a=(x_a)^2$.
%The reason we need $\XhD$ is it is easier for us to describe the map $\vpD$ using $\XhD$.

\begin{remark}
In  \cite{Le:tocome} we  extend  Theorem \ref{r.Muller} to the case when the marked surfaces have interior
marked points (or punctures), or some of boundary components of $\Sigma$ do not have marked points. The
advantage of having boundary components without marked points is that we can build a surgery theory, and
can alter the topology of the surfaces.
\end{remark}

\subsection{Flips of  triangulations} Let us recall the flip of a triangulation at an inner edge. Suppose
$\D$ is a  triangulation  of $\SM$ and $a$ is an inner edge of $\D$. There is a unique (up to $\cP$-isotopy)
$\cP$-arc $a^*$ different from $a$ such that $\D' = \D \setminus \{a\} \cup \{ a^*\}$ is a  triangulation,
and we call $\D'$
{\em the flip of $\D$ at $a$}.
\FIGc{mutation}{Flip at $a$.}{2cm}

One can obtain $a^*$ is as follows. The two triangles, each  having $a$ as an edge, together form a $\cP$-quadrilateral, with $a$ being one of its two diagonals, see Figure \ref{fig:mutation}.
Then $a^*$ is the other diagonal. The  edges $b,c,d,e$ in Figure \ref{fig:mutation} are not necessarily pairwise distinct. If they are
not, then either $b=d$ or $c=e$ (but not both) as  all other cases are excluded because $\cP\subset \pS$.

It is known that for any two  triangulations are related by a sequence of flips \cite{FST}.

\subsection{Coordinate change} Suppose $\D, \D'$ are two  triangulations of $(\Sigma, \cP)$. We have the
following algebra isomorphisms (skein coordinate maps)
$$ \vpD: \ttS \overset \cong \longrightarrow \tiX(\D), \quad \vp_{\D'}: \ttS \overset \cong
\longrightarrow \tiX(\D').$$
The {\em coordinate change map} $ \Phi_{\D,\D'}: \tiX(\D') \to \tiX(\D)$ is defined by
$\Phi_{\D,\D'}:= \vpD \circ (\vp_{\D'})^{-1}$, which is an $\cR$-algebra isomorphism.

\begin{proposition} \lbl{r.42}

(a) The coordinate change isomorphism $\Phi_{\D,\D'}$ is natural. This means,
 $$\Phi_{\D, \D}= \Id, \quad \Phi_{\D, \D''} = \Phi_{\D, \D'} \circ \Phi_{\D', \D''}.$$

 (b) The maps $\vp_\D: \cS \embed \tiX(\D)$ commute with the coordinate change  maps, i.e.
$$ \vp_{\D} = \Phi_{\D,\D'}  \circ \vp_{\D'}.$$

 (c) Suppose $\D'$ is obtained from $\D$ by a flip at an edge  $a$, with $a$ replaced by $a^*$ as in
 Figure \ref{fig:mutation}. Then $\Phi_{\Delta\Delta'}(X_v)= X_v$ for any $v \in \D \setminus \{a^*\}$
 and, with notations of edges as in Figure \ref{fig:mutation},
\be
\lbl{eq.aa1}
 \Phi_{\Delta\Delta'}(X_{a*}) =    [X_cX_e X_{a}^{-1}] + [ X_bX_dX_{a}^{-1}].
 \ee

\end{proposition}
\FIGc{proof.mutation}{Proof of \eqref{eq.aaa} }{2.2cm}
\begin{proof}
Parts (a) and (b) follow right away from the definition. Let us prove (c).  It is clear that
$\Phi_{\Delta\Delta'}(X_v)= X_v$ for any $v \in \D \setminus \{a^*\}$.
In $\cS$, using the skein relation (see Figure \ref{fig:proof.mutation}),  we have
$$
 a a^* = q ce  +   q^{-1} bd.
 $$

Multiplying $a^{-1}$ on the left,
\be \lbl{eq.aaa}
a^* = q^{\ve_1} [ce a^{-1}] + q^{\ve_2}[ bd a ^{-1}],
\ee
where $\ve_1, \ve_2 \in \BZ$. A careful calculation of $\ve_1$ and $\ve_2$
using the commutations between $a,b,c,d,e$ will show that $\ve_1=\ve_2=0$, and we get \eqref{eq.aa1}.
Another way to show $\ve_1=\ve_2=0$ is the following. The two monomials  $[ce a^{-1}]$ and  $[ bda^{-1}]$
are distinct. In fact, if they are the same, then either $c=b$ or $c=d$, which is impossible because then
either $\D$ or $\D'$ has a self-folded triangle. By Lemma \ref{r.reflection},  $\ve_1=\ve_2=0$.
\end{proof}

\begin{remark}
 One advantage of $\sX(\D)$ over the Chekhov-Fock algebra (see Section \ref{sec:shear}) is the coordinate
 change  maps come along naturally and are easy to study.
\end{remark}

\def\oC{\mathring C}

\subsection{Image of $\D$-simple knots under $\vpD$} Suppose $\D$ is a triangulation of $\SM$ and $\al$ is a $\cP$-arc or $\cP$-knot. We say that $\al$ is
 {\em $\D$-simple} if $\mu(\al,a)\le 1$ for all $a\in \D$.

 Suppose $\al$ is a $\D$-simple. After an isotopy we can assume that $\al$ is $\D$-normal, i.e. $\mu(\al,a)$ is equal to the number of internal common points of $\al$ and $a$, for all $a\in \D$.
Let $\cE(\al,\D)$ be the set of all edges $e$ in $\D$ such that $\mu(\al,e)\neq 0$ , and
$\cF(\al,\D)$ be the set of all triangles $\tau$ of $\D$   intersecting the interior of $\al$.
It is clear  that $\cE(\al,\D) \subset \oD$.
%If $a\in \cE(\al,\D)$, then $a$ is the common edge of exactly 2 triangles. %If $\tau\in \cF(\al,\D)$ then $\al$ intersects exactly 2 edges of $\tau$.

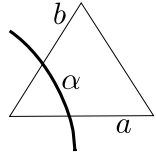
\FIGc{forbidden1}{Non-admissible case: $C(a)=-1, C(b)=1$.}{2cm}
{\em A coloring of $(\al,\D)$} is a map $C \in \BZ^\oD$ such that $C(e) =0$ if $ e \not \in \cE(\al, \D)$
and $C(e)\in \{ 1, -1\}$ if $e\in \cE(\al, \D)$.
A coloring $C$ of $(\al,\D)$ is said to be {\em admissible} if
for any triangle $\tau \in \cF(\al,\D)$ intersecting $\al$ at two edges $a$ and $b$, with notations of edges as in Figure \ref{fig:forbidden1}, one has
$(C(a), C(b)) \neq (-1,1).$
Denote by $\Col(\al,\D)$ the set of all admissible colorings of $\al$.

\def\oC{\mathring C}

\def\deq{\overset \bullet =}
\begin{theorem}
\lbl{r.fc}
Suppose $\D$ is a triangulation of a marked surface $\SM$ and
 $\al$ is a $\D$-simple, $\D$-normal knot. Then for any $C\in \Col(\al,\D)$, $CH$ has even entries, and
\be
\lbl{eq.fc}  \varphi_{\D}(\al) = \sum_{C\in \Col(\al,\D)} x^{C  H}.
\ee
\end{theorem}
 Recall that $H$ is the $\oD\times \D$ submatrix of the face matrix $Q=Q_\D$,
  and $CH$ is the matrix product where $C$ is considered as a row vector, and
$x^{C  H}\in \XhD$ is the normalized monomial.
\begin{proof}To simplify notations we identify $\cS$ with its image under the embedding $\vp_\D: \cS
\embed \sX(\D)$.
Thus, for $a \in \D$, we have $a= X_a= x_a^2$.

{\em Step 1.} Let $E= \bigcup_{e\in \cE(\al,\D) }e $, which is a $\cP$-link and will be considered as an
element of $\cS\subset \XD$. We express explicitly the product  $\al E$ as a sum of monomials as follows.
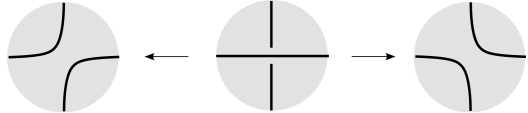
\FIGc{smoothing}{Smoothing of crossing, $+1$ on the left, $-1$ on the right}{1.5cm}

Each $e \in \cE(\al,\D)$  intersects $\al$ at exactly one point. Let $L$ be the $\cP$-link diagram
$\al \cup E,$
with $\al$ above all the $e \in \cE(\al,\D)$. Then $L$ represents the product
$\al E,$ which can be calculated by resolving all the crossings of $L$ using the skein relation. Each
crossing of $L$ has two  smoothings, the $+1$ one and the $-1$ one, see Figure \ref{fig:smoothing}.
Each coloring $C$ of $(\al,\D)$ corresponds to exactly one smoothing of all the crossings of $L$ by the
rule:  the smoothing of the crossing on the edge $e$ is of type $C(e)$. Let $L_C$ be the result of
smoothing all the crossings of $L$ according to $C$.
Then, with $ \|C\|= \sum_{e} C(e)$,
$$
 \al E = \sum_{C} q^{\| C \| } L_C,
$$
where the sum is over all colorings $C$ of $(\al,\D)$.
 This implies
 \be\lbl{eq.62}
 \al  = \sum_{C } \al_C, \quad \text{where } \
 \al_C= q^{\| C \| } L_C\,  E^{-1}% \beq L_C \, E^{-1}.
 \ee

{\em Step 2.}
Note that $L_C$  is a collection of  $\cP$-arcs, with exactly one in each triangle $\tau\in \cF(\al,\D)$,
see Figure \ref{fig:proalpha}.
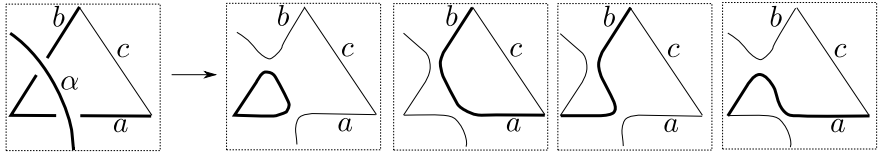
\FIGc{proalpha}{Four resolutions of left box corresponding to $(C(a),C(b))=(-1,1)$, $(1,-1)$, $(-1,-1)$,
and $(1,1)$}{2cm}

Denote the  $\cP$-arc of $L_C$ in $\tau$ by $\tw(\al,\tau,C)$, which is described in
Figure~\ref{fig:proalpha}:
\be
\lbl{eq.43}
\tw(\al,\tau,C) =
\begin{cases}
0 \quad &\text{if } C(a)=-1, C(b) = 1 \\
c &\text{if } C(a)=1, C(b) = -1 \\
b  &\text{if } C(a)=-1, C(b) = -1 \\
a  &\text{if } C(a)=1, C(b) = 1.
\end{cases}
\ee
The first identity of \eqref{eq.43} shows that $L_C=0$ if $C$ is not admissible. Hence \eqref{eq.62}
becomes
\be\lbl{eq.62aa}
 \al  = \sum_{C\in \Col(\al,\D) } \al_C, \quad \text{with } \
 \al_C\beq L_C\,  E^{-1},
 \ee
where  $u\beq v$ means $u= q^r v$ for some $r\in \BQ$. By construction,
\be
\lbl{eq.62d}
L_C\beq \prod_{ \tau \in \cF(\al,\D) } \tw(\al,\tau,C).
\ee

{\em Step 3.}
For each $\tau\in \cF(\al,\D)$ with notations of edges  as in Figure \ref{fig:forbidden1} let
\be \lbl{eq.43a}
E_\tau := x_a x_b.
\ee  Since $a= x_a^2$ and each $\tau\in \cF(\al,\D)$ has two edges intersecting $\al$, we have
\be
\lbl{eq.42a}
E \beq    \prod_{ e\in \cE(\al,\D) }e \beq  \prod_{ e\in \cE(\al,\D) }(x_e)^2  \beq \prod_{\tau\in
\cF(\al,\D)} E_\tau.
 \ee
Using \eqref{eq.62aa}, \eqref{eq.62d}, and \eqref{eq.42a}, we have
 \be \lbl{eq.62b}
\al_C \beq  L_C \, E ^{-1} \beq  \prod_{ \tau \in \cF(\al,\D) }  \tw(\al,\tau,C)  \,  E_\tau^{-1}  .
 \ee

\def\hC{\hat C}
{\em Step 4.} Let $\hat C\in \BZ^\D$ be the zero extension of $C\in \BZ^{\oD}$. Then $\hat C Q = CH $.
Using the explicit formula \eqref{eq.Qtau} of $Q_\tau$ and  \eqref{eq.43},  \eqref{eq.43a}, one can verify that
$$  \tw(\al,\tau,C)  \,  E_\tau^{-1} \beq  x ^{\hC   Q_\tau}.$$
Hence, \eqref{eq.62b} implies
\be
\lbl{eq.65a}
\al_C \beq \prod_{\tau \in \cF(\al,\D)} x ^{\hC   Q_\tau}\beq x ^{\hC   Q} = x^{C H}.
\ee
%By Lemma \ref{r.57}, we have $\vp_{\D, C}(\al) \beq\psi(\vk_C(\al))$.
{\em Step 5.} From \eqref{eq.62aa} and \eqref{eq.65a}, we have
$$ \al = \sum_{C \in \Col(\al,\D)} q^{f(C)}x^{C   H},$$
with $f_C \in \BQ$. Since $\rk(H) =|\oD|$ (by Lemma \ref{r.51a}),  $C  H$ are distinct when $C$ runs the set $\Col(\al,\D)$.  Because $\al$ is reflection invariant, Lemma
\ref{r.reflection} shows that $f_C=0$ for all $C \in \Col(\al,\D)$. This proves~\eqref{eq.fc}, as equality in
$\XhD$.

{\em Step 6.} Equation \eqref{eq.fc}  and Lemma \ref{r.basesub} shows that each
$x^{CH}$ is in $\XD$, which is equivalent to $CH$ has even entries. This completes the proof of the
theorem.
\end{proof}

\begin{remark}
The fact that $CH$ has even entries can be proved directly easily. A more general fact is proved in Lemma \ref{r.even1} below.
\end{remark}

\subsection{Triangulation-simple knots}  Theorem \ref{r.fc} gives the image under $\varphi_\D$ of $\D$-simple knots, but not all $\ooS=\cS(\Sigma, \emptyset)$. Following is the reason why in many applications this should be enough.

A knot $\al \in \Sigma$
is {\em triangulation-simple} if there is a triangulation $\D$ such that $\al$ is  $\D$-simple.

\begin{proposition} \lbl{r.generator} Suppose $\SM$ is a triangulable marked surface.
Then the algebra $\ooS$ is generated by the set of all triangulation-simple knots.
\end{proposition}
\begin{proof} Let $\D$ be a triangulation of $\SM$ and
 $\Gamma \subset \oD$ be a maximal subset such that by splitting $\Sigma$ along edges in $\Gamma$ one gets a disk. The resulting disk, denoted by
 $\bar \Sigma$,   inherits a triangulation $\bD$ from  $\D$.
 According to  \cite[Appendix A]{Muller} ( see also \cite{Bullock}), $\ooS$ is generated by knots in $\oS$ which meet each edge in $\Gamma$ at
 most once. Such a knot $\al$, when cut by $\Gamma$, is a collection of non-intersecting intervals,
 called $\al$-intervals, in the polygon $\bar \Sigma$. Each $\al$-interval has end points on boundary edges of $\bD$, called the
 ending edges of the interval.
No two different intervals have a common ending edge. After an isotopy we  can assume that  the edges of
$\bD$ and all $\al$-intervals are straight lines on  the plane. For an $\al$-interval, the convex hull of
its ending edges is either a triangle (when the two ending edges have a common vertex) or a quadrilateral,
called the hull of the interval. The hulls of two different $\al$-intervals do not have interior
intersection. For each hull which is a quadrilateral choose a triangulation of it by adding one of its
diagonals. Let $\D'$ be  any triangulation of $\Sigma$ extending the triangulations  of all of the hulls. Then
$\al$ is $\D'$-simple. This proves the proposition.
 \end{proof}

\subsection{Image of $\D$-simple arcs} Suppose $\al$ is a { $\D$ simple}, $\D$-normal  $\cP$-arc. We assume $\al \not\in \D$.

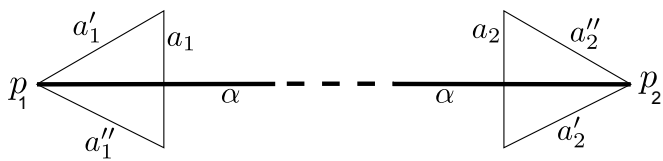
\FIGc{arctoskein}{A $\D$-simple arc $\al$.
%and its terminal triangles.
One might have $p_1=p_2$, and one of $a_1', a_1''$
might be one of $a_2', a_2''$}{2cm}

Then, with notations of edges as in Figure \ref{fig:arctoskein}, one has
\be
\lbl{eq.arctoskein}
\al = \sum_{C\in \Col(\al, \D)} \left[ x^{C Q}  \, x_{a'_1} x_ { a''_1} (x_{a_1})^{-1}\, x_{a'_2} x_ {
a''_2} (x_{a_2})^{-1}  \right]
\ee

The proof is a simple modification of that of Theorem \ref{r.fc}, taking into account what happens near the
end points of $\al$, and is left for the dedicated reader. We will not need this result in the current
paper.
%\subsection{Image of non $\D$-simple knot}

\def\oA{\overset \circ A}
\def\fF{\mathfrak F}

%%%%%%%%%%%%%%%%%%%%%%%%%%%%%%%%%%%%%%%
%%%%%%%%%%%%%%%%%%%%%%%%%%%%%%%%%%%%%%

\section{Chekhov-Fock algebra of marked surfaces and shear coordinates}
\lbl{sec:shear}

 In this section we show how the quantum trace map of Bonahon and Wong can be recovered from the natural embedding $\varphi_\D:\cS(\Sigma,\emptyset) \embed \XD$ by the shear-to-skein map and give an intrinsic description of the quantum Teichm\"uller space, for the case when $\SM$ is a triangulated marked surface.

 Throughout this section we fix a triangulable marked surface $\SM$; $\D$ will be a triangulation of $\SM$. We use notations $\ooS=\cS(\Sigma,\emptyset)$ and $\cS= \cS\SM$.

\subsection{Chekhov-Fock algebra and its square root version} \lbl{sec.CF} Here we define the Chekhov-Fock algebra $\YtD$ and its square root version $\YeD$ mentioned in Introduction.

In this subsection we allow $\SM$ to be more general, namely $\SM$ is a triangulated {\em generalized} marked surface, with triangulation $\D$.
The face matrix $Q=Q_\D\in\Mat(\D\times\D,\BZ)$ is defined (see Section \ref{sec.51}), but the vertex matrix $P$ cannot not be defined if there are interior marked points.

Recall that $\oQ$ is the $\oD \times \oD$
sub matrix of the face matrix $Q=Q_\D$.
Let $\cY(\D)$ be the quantum torus $\bT(\oQ ,q^{-1},y)$, i.e.
$$ \cY(\D) = R \la y_a^{\pm1}, a \in \oD  \ra /( y_a y_b = q^{- \oQ(a,b)} y_b y_a).$$
%The number of basis variables is $|\oD|$, which is less than that of $\sX(\D)$, which has $|\D|$ basis variables.
Let $Y_a= y_a^2$, for $a\in \oD$. Then the subalgebra $\YtD\subset \YD$ generated by
$Y_a^{\pm1}$ is the quantum torus $\bT(\oQ ,q^{-4},Y)$ with basis variables $Y_a, a \in \oD$. Let $\tYtD$
and $ \tYD$ be respectively the skew fields of $\YtD$ and $\YD$.
The preferred bases of $\YD$ and $\YtD$ are respectively $\{ y^\bk \mid \bk \in \BZ^{\oD }\}$ and $\{
y^\bk \mid \bk \in (2\BZ)^{\oD }\}$.

An element $\bk\in \BZ^\oD$ is called  {\em $\D$-balanced} if $\bk(\tau)$ is even for any
triangle $\tau\in \cF(\D)$. Here
$\bk(\tau)=\bk(a) + \bk(b) + \bk(c)$, where $a,b,c$ are edges of $\tau$, with the understanding that $\bk(e)=0$ for any boundary edge $e$.
Let $\cY^\ev(\D)$ be the $\cR$-submodule of $\cY(\D)$ spanned by $y^\bk$ with balanced $\bk$.
Clearly $\cY^\ev(\D)$ is an $\cR$-subalgebra of $\YD$ and $\YtD \subset \cY^\ev(\D)$.

%Clearly $\YtD \subset \YeD$.
%Another class of elements of $\cY^\ev$ is the following.
For a $\cP$-knot $\al$, let $\bk_\al\in \BZ^\D$ be  defined by $\bk_\al(e)= \mu(\al,e)$, which is clearly $\D$-balanced. Recall that $\al$ is $\D$-simple if $\bk_\al(e) \le 1$ for all $e\in \D$.

\begin{lemma}
\lbl{r.gen.Ye}
As an $\cR$-algebra, $\cY^\ev(\D)$ is generated by $\cY^{(2)}(\D)$ and $y^{\bk_\al}$, with all $\D$-simple knots~$\al$.

%(b) One has $\tYeD \cap \cY(\D)= \cY^\ev(\D)$.
\end{lemma}

\def\bu{{\mathbf u}}
\begin{proof}  A version of this statement already appeared in \cite{BW0}, and we use same proof.
Every $\bk \in \BZ^\oD$ has a unique presentation
$ \bk= \bu + \bm$
where $\bm\in (2\BZ)^{\oD }$ and $\bu \in \{0,1\}^\oD$. Hence,
\begin{align}
%\lbl{eq.13a}
\notag
\YD &= \bigoplus_{\bu \in \{0,1\}^\oD} y^\bu \YtD \\
\lbl{eq.13}
\cY^\ev(\D) &= \bigoplus_{\bu \in (\{0,1\}^\oD)_\ev} y^\bu \YtD,
\end{align}
where $(\{0,1\}^\oD)_\ev$ is the set of $\D$-balanced $\bu\in \{0,1\}^\oD$. Note that $(\{0,1\}^\oD)_\ev$
is naturally isomorphic to the homology group $H^1(\Sigma,\BZ/2)$.
For every $\bu\in (\{0,1\}^\oD)_\ev$, there exists $\al= \sqcup_{i=1}^j \al_i$ such that each $\al_i$ is a
$\D$-simple knot and $\bu = \sum \bk_{\al_i}$. Hence, \eqref{eq.13} shows that $\cY^\ev(\D)$ is generated
by $\YtD$ and $y^{\bk_\al}$ with all $\D$-simple knots $\al$.
\end{proof}

\def\hbk{{\hat \bk}}

Let $\QQ$ be the $\oD\times \D$ submatrix of $Q$.
 \begin{lemma}
 \lbl{r.even1} Suppose $\pS\neq \emptyset$ and
 $\bk\in \BZ^\oD$. Then $\bk H$ has even entries if and only if $\bk$ is $\D$-balanced.
 \end{lemma}
 \begin{proof} Let $\hat \bk\in \BZ^\D$ be the 0 extension of $\bk\in \BZ^\oD$. Then  $\bk H= \hat \bk Q$.
  %For a triangle $\tau\in \cF(\D)$ let $\hbk(\tau)=\hbk(a)+ \hbk(b) + \hbk(c)$, where $a,b,c,$ are edges of $\tau$.
 One has
 $$ \bk H= \hat \bk Q= \sum_{\tau \in \cF(\D)} \hbk Q_\tau.$$
 Note that $Q_\tau(a,b)\neq 0$ only if $a, b$ are edges of $\tau$.

 (i)
 Suppose $\tau\in \cF(\D)$ has edges $a,b,c$ with $c$ a boundary edge. Then $\hbk(c)=0$.
 Since $\tau$ is the only triangle having $c$ as an edge,
 \begin{align*}
 \hbk Q(c)= (\hbk Q_\tau)(c) \equiv \hbk(a)+ \hbk(b)\equiv \hbk(\tau) \pmod 2,
 \end{align*}
where the second equality follows from Lemma \ref{r.Qtau0}.
% $\bk H$ is even at every boundary edge if and only if $\hbk(\tau)$ is even at every triangle having a boundary edge as an edge.

 (ii) Suppose $a\in \oD$, with $\tau, \tau'\in \cF(\D)$ the two triangles having $a$ as an edge. Then, again using Lemma \ref{r.Qtau0},
 $$(\hbk Q)(a) = (\hbk (Q_\tau + Q_{\tau'}))(a)\equiv
  \hbk(\tau) + \hbk(\tau') \pmod 2.$$

 Since  $\D$ has at least one boundary edge and $\Sigma$ is connected,
 (i) and (ii) show that $\hbk Q$ has even entries if and only if $\hbk(\tau)$ is even for any $\tau\in \cF(\D)$, or equivalently, $\bk$ is $\D$-balanced.
 \end{proof}

\begin{remark} If we use the face matrix $Q$ instead of its submatrix $\oQ$, then $\YtD$ is the
Chekhov-Fock algebra defined in \cite{BW0,Liu}, and $\YeD$ is the Chekhov-Fock square root algebra of
\cite{BW0}.
The skew field $\tY^{(2)}(\D)$  is considered as a {\em quantization } of a certain version of the
Teichm\"uller space of $\SM$, using the {\em shear coordinates}, see \cite{CF,BW0}.
\end{remark}

\subsection{Shear-to-skein map} \lbl{sec.psi} From now until the end of  this section we fix a triangulable  marked surface $\SM$ and use the notations  $\ooS= \cS(\Sigma, \emptyset)$ and  $\cS=\cS\SM$.
Suppose $\D$ is a triangulation of $\SM$, and $P$ and $Q$ are respectively its vertex matrix and face matrix.
Recall that $H$ is the $\oD\times \D$ submatrix of $Q$.
By Lemma \ref{r.51a},  $\rk(H)= |\oD|$ and
\be
\lbl{r.51a1}\QQ P \QQ^\dag = -4 \oQ.
\ee
 Hence,
Proposition \ref{r.homo}  shows that there is a unique injective $\cR$-algebra homomorphism
$$ \psi: \cY(\D) \to \sX^{\left( \frac 12 \right)}(\D),$$
which is a multiplicatively linear homomorphism, such that
\be
\lbl{eq.44a}
\psi(y^\bk) = x^{\bk H}.
\ee

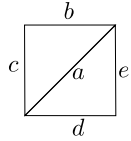
\FIGc{shear-to-skein}{Edges of a quadrilateral with a diagonal}{1.8cm}
We call $\psi$ the {\em shear-to-skein map}. Explicitly, if $a,b,c,d,e$ are
edges of $\D$ as in Figure \ref{fig:shear-to-skein}, then
\be
\lbl{eq.shear}
\psi(y_a) = [ x_b x_c^{-1}  x_d x_e^{-1}].
\ee

It is the factor $4$ in equation \eqref{r.51a1} that forces us to enlarge $\XD$ to $\XhalfD$ to accommodate the images of $\psi$. While $\XD$ has a geometric interpretation coming from skein, $\XhalfD$ does not and is only  convenient for algebraic manipulations. It turns out that $\cY^\ev(\D)$ is exactly the subset of $\cY(\D)$ whose image under $\psi$ is in $\XD$, which explains how $\YeD$ arises naturally in the framework of the shear-to-skein map.

\def\hbk{\hat \bk}
\begin{proposition}   \lbl{r.Ybl}In $\psi: \cY(\D)\to \XhalfD$,
one has $\psi^{-1}(\XD)= \cY^\ev(\D)$.
\end{proposition}
\begin{proof} Recall that
 $\{ y^\bk \mid \bk \in \BZ^\oD\}$ and  $\{ x^\bm \mid \bm \in (2\BZ)^\D\}$ are respectively $\cR$-bases of   $\cY(\D)$ and $\XD$, and $\psi(y^\bk)= x^{\bk H}$. Hence
 $\psi^{-1}(\XD)$ is $\cR$-spanned by all $y^\bk$  such that $\bk H$ has even entries.
  Since marked surface has non-empty boundary, Lemma \ref{r.even1} shows that $\bk H$ has even entries if and only if $\bk$ is $\D$-balanced.  Hence, $\psi^{-1}(\XD)= \cY^\ev(\D)$.
\end{proof}

\begin{remark}
When $q=1$, formula \eqref{eq.shear}  expresses the shear coordinates in terms of the Penner
coordinates. We are grateful to F. Bonahon for suggesting that relation between the Chekhov-Fock algebra
and the Muller algebra, on the classical level, should be the relation between the
 shear coordinates and the Penner coordinates.
\end{remark}

\subsection{Change of shear coordinates} The shear-to-skein map $ \psi: \cY^\ev(\D) \to \XD$
extends to a unique algebra homomorphism $\tpsi: \tY^\ev(\D) \to \tXD$, which is also injective. Here
$$\hYeD:=
\YeD \tYtD,$$
which is  an $\cR$-subalgebra of $\tYD$.
Let
$ \tpsi_\D: \tY^\ev(\D) \embed \tcS$
be the composition
$$ \tY^\ev(\D) \overset{\tpsi}{\longrightarrow} \tXD  \overset{\cong }{\underset{(\varphi_\D)^{-1}}\longrightarrow}\tcS.$$
\begin{theorem}
\lbl{r.shearchange}
 (a) The image $\tcS^\ev:= \tpsi_\D(\tY^\ev(\D)) \subset \tcS$ does not depend on the triangulation $\D$. Similarly, $\tcS^{(2)}:= \tpsi_\D(\tY^{(2)}(\D))$ does not depend on $\D$.

(b) Suppose  $\D,\D'$ are two different triangulations of $\SM$. Then the algebra isomorphism $\Theta_{\D \D'}: \hYeDp \to \hYeD$, defined by $\Theta_{\D \D'}=  \tpsi_{\D'} \circ (\tpsi_\D)^{-1}$, coincides with shear
coordinate change map defined in  \cite{Hiatt,BW0}. Here $(\tpsi_\D)^{-1}$ is defined on $\tcS^\ev$.

(c) The image  $\tcS^{(2)}= \tpsi_\D(\tY^{(2)}(\D))$ is the sub-skew-field of $\ttS$ generated by all elements of the form $a b^{-1} c d^{-1}$, where $a,b,c,d$ are edges in a cyclic order of a $\cP$-quadrilateral.
\end{theorem}
For the definition of $\cP$-quadrilateral, see Section \ref{sec.42}.
The restriction of $\Theta_{\D \D'}$ onto $\tY^{(2)}(\D')$  is an isomorphism from $\tY^{(2)}(\D')$ to
 $\tYtD$ and is the coordinate change map constructed earlier by Chekhov-Fock and Liu \cite{CF,Liu}.
The proof of Theorem \ref{r.shearchange} is not difficult: it consists mainly of calculations, a long definition of
Hiatt's map, and will be presented in Appendix.

From the construction, we have the following commutative diagram
\be
\lbl{eq.dia9}
 \begin{CD}  \hYeDp   @> \tpsi>> \hXDp \\
 @V
 \Theta_{\D\D'}  VV  @V V \Phi_{\D\D'} V    \\
  \hYeD   @> \tpsi>> \hXD .
\end{CD}
\ee

\begin{remark} The quantum Teichm\"uller space \cite{CF} is defined abstractly by identifying all $\tY^{(2)}(\D)$ via the coordinate change maps $\Theta_{\D \D'}$. Theorem \ref{r.shearchange} allows to  realize the quantum Teichm\"uller space as
a concrete subfield $\ttS^{(2)}$  of the skein skew field $\ttS$, which does not depend on the triangulation.
\end{remark}

\newcommand{\LHS}[1]{\text{LHS of #1}}
\newcommand{\RHS}[1]{\text{RHS of #1}}

\subsection{Quantum trace map and proof of Theorem \ref{thm.1}} Recall that $\ooS= \cS\SM$.
In \cite{BW0}, Bonahon and Wong construct the {\em quantum trace map}, an injective algebra  homomorphism,
$$ \tr_q^\D: \ooS \to \YeD,$$
%called the quantum trace map. When $q=1$, the map reduces to the usual trace map in Teichm\"uller theory.
which is natural with respect to the shear coordinate change, i.e.  diagram \eqref{eq.dia6} is commutative:
\be
\lbl{eq.dia6}
 \begin{CD}  \ooS   @> \tr_q^{\D'}>> \hYeDp \\
 @V
  \id  VV  @V V \Theta_{\D\D'} V    \\
  \ooS   @> \tr_q^\D >> \hYeD .
\end{CD}
\ee
The construction of $\tr_q$ involves many difficult calculations and the way $\tr_q^\D$ was constructed
remains a mystery for the author. Here we show that the quantum trace map is $\vpD : \ooS \to \XD$, via the shear-to-skein map. The following is Theorem \ref{thm.1} of Introduction.

\begin{theorem} \lbl{r.52} Let $\SM$ be a triangulated marked surface with triangulation $\D$, and $\ooS= \cS(\Sigma, \emptyset)$.

(a) In the diagram
\be
\lbl{eq.2heads1}
 \ooS \overset {\vp_\D}\lembed\XD\overset {\psi} {\rembed} \YeD,
 \ee
the image of $\vp_\D$ is contained in the image of $\psi$, i.e.
\be
\lbl{eq.incl1z}
\vp_\D(\ooS  ) \subset \psi (\YeD).\ee

(b) The algebra homomorphism $  \vk_\D: \ooS \to \hYeD$, defined by $\psi^{-1}\circ \vp_\D$, coincides with the quantum trace map of Bonahon and
Wong \cite{BW0}.
\end{theorem}

\begin{proof} (a) Recall that $\tpsi: \hYeD\to \hXD$ is the natural extension of $\psi$. \\
{\em Step 1.}  Let
$\al$ be a $\D$-simple knot. First by Theorem \ref{r.fc} then by \eqref{eq.44a}, we have
\begin{align}
\vp_\D(\al) &= \sum_{C \in \Col(\al,\D)} x^{C   H}= \sum_{C \in \Col(\al,\D)} \psi (y^{C})
 \lbl{eq.82}.
\end{align}
Note that $y^C \in \cY^\ev$ for each $C \in \Col(\al,\oD)$. Hence,
\eqref{eq.82} implies $\vp_\D(\al) \in \psi(\cY^\ev)$.

{\em Step 2.}
Now assume $\al$ is a triangulation-simple knot, i.e. there is another triangulation $\D'$ such that
 $\al$ is $\D'$-simple. By Step 1,
  $$\vp_{\D'}(\al) \in \psi(\cY^\ev(\D')) \subset \tpsi(\hYeDp).$$
   The commutativity of
 Diagram~\eqref{eq.dia9} shows that $\vpD(\al) \in \tpsi(\hYeD)$.
Since $\ooS$ is generated by triangulation-simple knots (Proposition~\ref{r.generator}), we have a weaker version of \eqref{eq.incl1z}:
\be
\lbl{eq.incl1a}
\vp_\D(\ooS  ) \subset \tpsi (\hYeD).\ee

\def\YtwoD{\cY^{(2)}(\D)}
{\em Step 3.}
Let us now prove \eqref{eq.incl1z}. Due to \eqref{eq.incl1a} and $\vp_\D(\ooS)\subset \XD$, it is enough to show
 that
 $$ \tpsi(\tY^\ev(\D)) \cap \XD = \psi(\cY^\ev(\D)).$$

 Suppose $z\in \tY^\ev(\D)$ and $\tpsi(z)=z'\in \XD$, we need to show $z\in \cY^\ev(\D)$. Then $ z = u v^{-1}$, where $u\in \cY^\ev(\D)$ and $v\in \YtwoD$. We have
 \be
 \lbl{eq.123}
 \psi(u) = z' \psi(v).
 \ee
 All $\psi(u), z', \psi(v)$ are in $\XhalfD$, which has as a basis the set $\{x^\bk, \bk \in \BZ^\D\}$. The two $\psi(u), \psi(v)$ are in $\psi(\cY(\D))$, which has as a basis the set of all $x^\bk$, where $\bk$ runs over the subgroup $(\BZ^\oD)H$ of $\BZ^\D$. Using a total order of $\BZ^\D$ which is compatible with the addition (for example the lexicographic order) to compare the highest order terms in \eqref{eq.123}, we see that $z' \in \psi(\cY(\D))$. This means
  $$ z'= \sum c_i x^{\bk_i H},$$
 where $\bk_i \in \BZ^\oD$ and $c_i\in \cR$. Since $z'\in \XD$, Proposition \ref{r.Ybl} shows that each $\bk_i$ is balance. It follows that $z'= \psi(\sum c_i x^{\bk_i}) \in \psi(\cY^\ev(\D))$. This completes proof of part (a).

(b)  Formula \eqref{eq.82} shows that for any $\D$-simple knot $\al$,
\be
\lbl{eq.qt}
 \varkappa_D(\al) =\sum_{C\in \Col(\al,\D)} y^C  \ \in \cY^\ev(\D),
 \ee
 which is exactly $\tr_q^\D(\al)$, where $\tr_q^\D$ is the Bonahon-Wong quantum trace map, see  \cite[Proposition 29]{BW2}.
 Thus $\varkappa_\D=\tr_q$ on $\D$-simple knots. The commutativity \eqref{eq.dia9} and the naturality of
 the quantum trace map with respect to the shear coordinate change, Equation \eqref{eq.dia6},  then show
 that $\varkappa_\D=\tr_q^\D$ on triangulation-simple knots. Since triangulation-simple knots generate
 $\ooS$, we have $\varkappa_\D=\tr_q^\D$.
\end{proof}

 \begin{remark}
 We need only a special case of \cite[Proposition 29]{BW2}, namely, the case when $\al$ is $\D$-simple, and no cabling is applied to $\al$. Although the proof of \cite[Proposition 29]{BW2} has long calculations, this special case is much simpler and follows almost immediately from the definition of $\tr_q^\D$ in \cite{BW0}.
 \end{remark}

\def\G{{\Gamma}}
\def\hS{\widehat \Sigma}
\def\hD{{\hat \D}}
\def\pD{{\varphi_\D}}
\def\hP{{\widehat P}}
\def\hal{\hat \al}
\def\hQ{{\widehat Q}}

\section{Generalized marked surface, quantum trace map}
\lbl{sec:generators}

Now we return to the case of {\em generalized} marked surface.
Throughout this section we fix a  triangulated generalized marked surface
$\SM$, with   triangulation $\D$.
We use the notation $\ooS:=\cS(\Sigma\setminus \cP, \emptyset)$. We describe a  set of generators for the algebra $\ooS$ and calculate their values under the quantum trace map.

\subsection{Generators for $\ooS$} \lbl{sec.state}
The following was proved in \cite[Lemma 39]{BW2} for the case $\pS=\emptyset$.
\begin{lemma}
The $R$-algebra $\ooS$ is generated by $\cP$-knots $\al$ such that $|\al\cap a|=1 $ for some $a\in \D$.
\end{lemma}
\begin{proof} Let $\Gamma\subset \oS$ be a maximal set having the property that $\Sigma \setminus \bigcup_{e\in \G} e$ is contractible.
In \cite[Appendix A]{Muller} (see
Lemma A.1 there and its proof) it was shown that
the set of all $\cP$-knots $\al$ such that $\mu(\al,a)\le 1$ for all $a\in \Gamma$ generates $\ooS$ as an $\cR$-algebra. If
 a $\cP$-knot $\al$ has $|\al\cap a|=0 $ for all $a\in \G$, then $\al$ is in the complement disk  $\Sigma \setminus \bigcup_{e\in \G} e$, and hence is trivial. It follows that the set of all $\cP$-knots $\al$ such that $\mu(\al,a)=1$ for some $a\in\Gamma$ generates $\ooS$.
\end{proof}

\subsection{States of $\D$-normal knot} \lbl{r.states}

Suppose $\al$ is a $\D$-normal knot, i.e.  it is non-trivial and  $\mu(\al,e)= |\al \cap e|$ for all $e\in \D$.
As usual, $\cE(\al,\D)$ is the set of all edges in $\D$ meeting $\al$ and
 $\cF(\al,\D)$ is
 the set of all triangles meeting $\al$. It is clear  that $\cE(\al,\D) \subset \oD$.
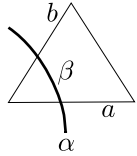
 \FIGc{forbidden2}{Forbidden pair (non-admissible case): $s(\beta\cap a)=-1, s(\beta\cap b)=1$.}{2cm}
\def\Col{\operatorname{St}}

 {\em A state of $\al$} with respect to $\D$ is a map $s: \al \cap E_\D \to \{1,-1\}$, where $E_\D= \bigcup _{e\in \D} e$.
Such a  state $s$  is called {\em admissible} if for every connected component $\beta$ of $\al\cap \tau$, where
  $\tau \in \cF(\al,\D)$, the values of $s$ on the end points of $\beta$ can not be the forbidden
 pair described in Figure \ref{fig:forbidden2}. Let $\Col(\al,\D)$ denote the set of all admissible states
 of  $\al$. For $s \in \Col(\al,\D)$ let $\bk_s \in \BZ^\oD$ be the function defined by
 $$ \bk_s(e) = \sum_{v \in \al \cap e} s(v).$$

\def\aone{{a_{(1)}}}
\def\atwo{{a_{(2)}}}
\def\htau{{\hat \tau}}
\def\hY{\hat \cY}
\subsection{Quantum trace of a knot crossing an edge once}\lbl{sec.73}
Bonahon and Wong \cite{BW0} constructed a quantum trace map
$ \tr_q^\D:\ooS \to \YeD$. Since $\cP$-knots crossing one of the edges of $\oD$ once generate $\ooS$, we want to understand the images of those under $\tr_q^\D$.

\begin{proposition}
\lbl{r.fc2}
Suppose $\al$ is a $\D$-normal $\cP$-knot and $|\al\cap a|=1$, where $a\in \oD$. One has
\be
\lbl{eq.22c}
\tr_q^\D(\al)= \sum_{s\in \Col(\al,\D)} q^{\uu(s)} y^{\bk_s},
\ee
where $\uu(s)\in \frac 12 \BZ$ is defined in Section \ref{sec.uu} below.
\end{proposition}
The proof is straight forward from the definition of $\tr_q^\D$, but first we have to prepare some definitions in Section \ref{sec.def1}--\ref{sec.uu}, and then prove the proposition in Section \ref{sec.def3}.

\subsection{Face matrix revisited} \lbl{sec.def1}
By splitting $\Sigma$ along all inner edges, we get $\hS$ which is the disjoint union of triangles $\htau$, one  for each triangle $\tau\in \cF(\D)$,
 with a gluing back map $\pr: \hS \to \Sigma$. Let $\hD$ be the set of all edges of $\hS$. Then $\pr: \hS \to \Sigma$ induces a map $\pr_*: \hD \to \D$, where
$(\pr_*)^{-1}(e)$ consists of the two edges splitted from $e$ for all $e\in \oD$.
We call $\htau$ the lift of $\tau$,  and a lift of $e\in \D$ is one of the edges in $(\pr_*)^{-1}(e)$.

All the edges and vertices of $\htau$ are distinct. Let $\hQ=Q(\hS,\hD)\in \Mat(\hD \times \hD,\BZ)$ be the face matrix of the disconnected triangulated surface $(\hS,\hD)$, i.e.
$\hQ= \sum_{\tau \in \cF(\D)} Q_\htau$, where $Q_\htau\in \Mat(\hD \times \hD,\BZ)$, with counterclockwise edges $a,b,c$,  is the 0-exention of the
 $\{a,b,c\}\times \{a,b,c\}$-matrix given by formula~\eqref{eq.Qtau}.

 The quantum torus $\hY:= \bT(\hQ, q^{-1}, y)$
has the set of basis variables parameterized by edges of all $\htau$. There is an embedding  (a  multiplicatively linear homomorphism of Lemma \ref{r.homo})
\be \lbl{eq.22d}
 \iota : \cY(\D) \embed \hY, \quad   \iota(y_e) = [y_{e'} y_{e''}],
\ee where $e', e''$ are lifts of $e$.

\def\ff{\mathfrak g}
\def\tal{\tilde \al}

\subsection{Definition of $\uu$} \lbl{sec.uu}
Fix an orientation of $\al$.
Let $V$ be the lift (i.e. the preimage under $\pr$) of  $ \bigcup_{e \in \D} (\al \cap e)$. For $v\in V$ let $e(v) \in \hD$ be the edge containing $v$, and $s(v)= s(\pr(v))$. The orientation of $\al$ allows us to define order on $V$ as follows.
The edges in $\oD$ cut $\al$ into intervals $\al_1,\dots,\al_k$, which are numerated so that if one begins at $\al\cap a$ and follows the orientation, one encounters $\al_1,\dots, \al_k$ in that order. If $\al_i\subset \tau\in \cF(\D)$, then $\al_i$ lifts to an interval $\tal_i\subset \htau$. Let $U_i=(u'_i,u''_i)$, where $u'_i,u''_i$ are  respectively be the beginning point and the ending point of $\tal_i$. Then we order $V$ so that
$$ u'_1 < u''_1 < u'_2 < u''_2 <\dots < u'_k < u''_k.$$

For $u,v\in V$ denote $u \ll v$ if $u < v$ and $(u,v)\neq U_i$ for all $i$.
For $\tau\in \cF(\D)$ let $V_\tau=V\cap \htau$. Define
\begin{align}
\lbl{eq.ff}
\ff(\tau;s) &:= -\frac 12\sum_{u,v\in V, \ u \ll v} Q_\htau(e(u),e(v)) s(u) s(v)
\\
&= -\frac 12\sum_{u,v\in V_\tau, \ u \ll v} Q_\htau(e(u),e(v)) s(u) s(v),
\end{align}
where the second equality holds since $Q_\htau(e(u),e(v))=0$ unless $u,v \in V_\tau$.
Define
\be
\lbl{eq.uu}
 \uu(s) := \sum_{\tau \in \cF(\D)} \ff(\tau;s).
 \ee

 \subsection{Proof of Proposition \ref{r.fc2}} \lbl{sec.def3}

\def\rprod{\operatorname{
\overrightarrow\prod}
}

From  \cite[Proposition 29]{BW2},
\be
\lbl{eq.22}
\iota\left( \tr_q^\D(\al)\right)= \sum_{s\in \Col(\al,\D)} z_0(s) z_1(s) \dots z_k(s),
\ee
where $z_i(s)= [(y_{e(u)})^{s(u)} (y_{e(v)})^{s(v)}]$, with $U_i= (u,v)$.
By definition of the normalized product,
\be
\lbl{eq.22a}
 \iota\left( \tr_q^\D(\al)\right)= \sum_{s\in \Col(\al,\D)} q^{\uu_1(s)} \rprod_{v\in V} (y_{e(v)})^{s(v)},
 %(y_0^+)^{s_0}
\ee
where $\rprod_{v\in V}$ is the product  in the increasing order (from left to right), and
 \begin{align}
 \lbl{eq.u1}
 2\uu_1(s)& =     \sum_{(u,v)\in \{U_1,\dots,U_k\}} \hQ(e(u), e(v)) s(u) s(v).
 \end{align}
By the definition of the normalized product,
\begin{align}
\rprod_{v\in V} ((y_{e(v)})^{s(v)}  &= q^{\uu_2(s)}\left [ \rprod_{v\in V} (y_{e(v)})^{s(v)}  \right ]
= q^{\uu_2(s)} \iota\left (y^{\bk_s} \right), \lbl{eq.22b}
\end{align}
where
\be
\lbl{eq.u2}
2\uu_2(s) = -\sum_{u,v \in V, \ u<v} \hQ(e(u), e(v)) s(u) s(v).
  \ee
  Using \eqref{eq.22b} in \eqref{eq.22a} and $\uu= \uu_1 +\uu_2$, we get
\eqref{eq.22c}. This completes the proof of Proposition \ref{r.fc2}.

\def\oC{\mathring C}

\def\deq{\overset \bullet =}

\def\hQ{{\hat Q}}

%\subsection{The case of marked surface}
\begin{corollary} Suppose the assumption of Proposition \ref{r.fc2}. Assume that $\SM$ is a marked surface, i.e. $\cP \subset \pS$. Then
 %$\bk_s H$ has even entries for all $s\in \Col(\al,\D)$, and
\be
\lbl{eq.m1}
\pD(\al)  =\sum_{s \in \Col(\al,\D)} q^{\uu(s)} x ^{\bk_s H}.
\ee
\end{corollary}
%Recall that $H$ is the $\oD\times \D$ submatrix of $Q$.
\begin{proof}
Because $\varphi_\D= \psi \circ \tr_q^\D$ by Theorem \ref{r.52}, Identity \eqref{eq.m1}
follows from \eqref{eq.22c}.
\end{proof}

\begin{remark}
Again  we need only a special case of \cite[Proposition 29]{BW2}  when  no cabling is applied to $\al$. This special case is very simple and follows almost immediately from the definition of $\tr_q^\D$ in \cite{BW0}.
 \end{remark}

\begin{remark}
 The definition of $\uu(s)$ a priori depends on the choice of an edge $a$ such that $\mu(\al,a)=1$ and an orientation of $\al$. This does not affects what follows.
\end{remark}

\begin{remark}% (a) The fact that $\bk_s H$ has even entries can also be proved directly easily.

% (b)
One can also directly prove Identity \eqref{eq.m1} without using Theorem \ref{r.52} by extending the calculation used in the proof Theorem \ref{r.fc}. This way we can get a new proof of Theorem~\ref{r.52} without using the change of basis maps $\Phi_{\D \D'}$.
\end{remark}

\def\YtL{\cY^{(2)}(\La)}
\def\YL{\cY(\La)}
\def\Ld{{\Lambda_\partial}}
\def\Dd{{\D_\partial}}
\def\oL{{\mathring \La}}

\def\bQ{\bar Q}

\section{Triangulated generalized marked surfaces}\lbl{sec:puncture}
In Section \ref{sec:shear} we showed that the quantum trace map of Bonahon and Wong can be recovered from the natural embedding $\varphi_\D: \ooS \to \XD$ via the shear-to-skein  map, for  triangulated marked surfaces. In this section we establish a similar result for  triangulated generalized marked surfaces. This includes the case of  a triangulated punctured surface without boundary, the original case considered in \cite{BW0} and discussed in Introduction.

Throughout this section we fix a triangulated generalized marked surface $(\bsS,\cP)$ with  triangulation $\La$.  This means $\bsS$ is a compact oriented connected surface with (possibly empty) boundary $\partial \bsS$, $\cP\subset \bsS$ is a finite set, and $\La$ is a $\cP$-triangulation of $\bsS$. Let $\oL\subset \La$ be the subset of inner edges, and $\Ld= \La \setminus \oL$ be the set of boundary edges. Let $\ocP$ be the set of interior marked points, i.e. $\ocP = \cP \setminus \partial \bsS$, and $\sS=\bsS\setminus \ocP$.

\def\bH{\bar H}
\def\bD{{\La}}
%\def\bQ{{\mathbullet{Q}}}
%\subsection{Chekhov-Fock algebra and its square root version}

The various versions of Chekhov-Fock algebras $\YtL \subset \YeL\subset \YL$ were defined in Section~\ref{sec.CF}.
Let us recall the definition of $\cY(\La)$ here. Let $\bQ$ be the $\oL\times \oL$ submatrix of the face matrix $Q_\La$.
Then  $\YL$ is the quantum torus $\bT(\bQ,q^{-1},z)$:
$$ \YL= \cR\la y_a^{\pm 1}, a \in \oL \ra /( y_a y_b = q^{- \bQ(a,b)} y_b y_a ).$$

\begin{remark} When $\partial \bsS =\emptyset$, our $\cY^{(2)}(\La)$ and $\cY^\ev(\La)$ are respectively  the Chekhov-Fock   algebra and
 Chekhov-Fock square root algebra $\cZ^\omega_\lambda$ of \cite{BW0},
with our $q,\Lambda$ equalling, respectively, $\omega^2,\lambda$ of \cite{BW0}.
\end{remark}

\subsection{Associated marked surface}\lbl{sec.ass}
For each interior marked point $p\in \ocP$ choose a small  disk $D_p\subset \bsS$ such that $p \in \partial D_p$.
Let $\Sigma$ be the surface obtained from $\bsS$ by removing the interior of all $D_p, p\in \ocP$.
We call $(\Sigma, \cP)$ the  {\em marked surface associated to the generalized marked surface $(\bsS,\cP)$}.

It is clear that $\cS(\Sigma,\emptyset)$ is canonically isomorphic to $\cS(\sS,\emptyset)$; the isomorphism is given by the embeddings
 $(\Sigma \setminus \ocP) \embed \sS$ and $(\Sigma \setminus \ocP) \embed \Sigma$ which induce  isomorphisms $ \cS(\Sigma,\emptyset) \cong \cS(\Sigma \setminus \ocP, \emptyset)\cong  \cS(\sS,\emptyset)$. We thus identify  $\cS(\Sigma,\emptyset)$ with $\cS(\sS,\emptyset)$, and use $\ooS$ to denote any of them. We also simply use $\cS$ to denote $\cS\SM$.

For each $p\in \ocP$ let $c_p$ be the boundary loop (which is $\partial D_p$) based  at $p$.  Note  that
each $c_p$ is in the center of $\cS=\cS(\Sigma, \cP)$.

 A triangulation of $\SM$ can be constructed beginning with $\La$,  as follows.
Since $\mu(c_p,a)=0$ for any $a\in \La$,
after an isotopy (of edges in $\oL$) we can assume that each $c_p$ does not intersect the interior of any $a\in \La$, i.e. the interior of each disk $D_p$ is inside some triangle of $\La$.
The set  $\La \cup \{  c_p \mid p \in \ocP\}$, considered as set of $\cP$-arcs in $\Sigma$, is not a maximal collection of pairwise non-intersecting and pairwise non-$\cP$-isotopic   $\cP$-arcs, and hence can be extended (in many ways) to a triangulation $\D$ of $\SM$.

Here is a  concrete construction of $\D$.
Suppose a triangle $\tau\in \cF(\La)$  contains   $k$ $D_p$'s. Here $k$ can be 0, 1, 2, or 3.
After removing the interiors of each $D_p$ from $\tau$ we get a $\cP$-polygon, and we add $k$ of its diagonals to triangulate it, creating $k+1$ triangles for $\D$. See Figure \ref{fig:triangulation} for the case $k=1$ and $k=2$,
where one of the many choices of adding diagonals is presented. Then $\D$ is obtained by doing this to all triangles of $\La$. We call $\D$ a {\em lift} of $\La$.

Every edge of $\D$, except for the $c_p$ with $p\in \ocP$, is $\cP$-isotopic in $\bsS$ to an edge in $\La$. Let $\omega: \D \setminus \{ c_p \mid p \in \ocP\} \to \La$ be the map defined by $w(a)$ is $\cP$-isotopic in $\bsS$ to $a$. For example, in  Figure \ref{fig:triangulation}, $\omega(d)=a$. Note that $w(a)=a$ if $a\in \Lambda$, i.e. $\omega$ is a contraction.

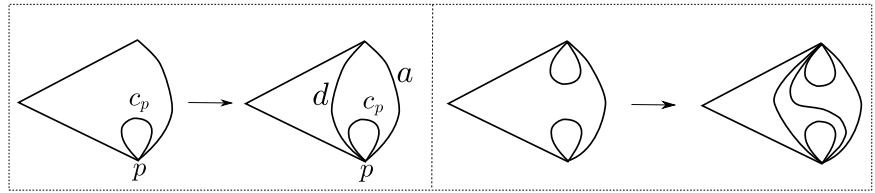
\FIGc{triangulation}{Adding diagonals to get a triangulation $\D$ of $\SM$: the case when $\tau$ has one $D_p$ (left) or two $D_p$'s (right). We have $\omega(d)=a$. The triangle with edges $a,d,c_p$ is a fake triangle, while the right picture has two fake triangles. }{2.5cm}

 The triangle of $\D$ having $c_p$ as an edge,  where $p\in \ocP$, is denoted by $\tau_p$ and is a called a {\em fake triangle}, see Figure \ref{fig:triangulation}.

\subsection{Skein coordinates}
Let $P\in \Mat(\D \times \D,\BZ)$ be the vertex matrix of $\D$ (defined in Subsection~\ref{sec.vmatrix}).
Recall that $\XD=\bT(P,q,X)$.
As a based $\cR$-module, $\XD$ has preferred base
 $\{ X^\bk, \bk \in \BZ^\D\}$.
Let $\bXD $ be the based $\cR$-submodule of $\sX(\D)$ with preferred base the set of all $X^\bk$ such that
$\bk(c_p) = 0$ for all $p\in \ocP$.
Let $\pi: \XD \to \bXD$ be the
canonical projection (see Section~\ref{sec.based}), which is an $\cR$-module homomorphism but not an
$\cR$-algebra homomorphism.

Recall that we have a natural embedding $\vp_\D: \cS \embed \sX(\D)$. To simplify notations in this
section we will
    identify $\cS$ with a subset of $\sX(\D)$ via $\vp_\D$. Thus, we have $a= X_a$ for any $a\in
    \D$.
{\em The skein coordinate map}
$\bvp_\D: \ooS \to \bXD$ is defined to be  the composition
\be
\bvp_\D: \ooS  \overset{\vp_\D}{\lembed} \XD \overset{\pi}{\longrightarrow}\bXD.
\ee
\begin{remark}  If $\cS'( \bar \fS,\cP)= \cS(\fS,\cP)/(c_p)$, where $(c_p)$ is the ideal generated by $c_p, p\in \cP$
then $\varphi_\D$ descends to a map $  \cS'( \bar \fS,\cP) \to
\bXD$, which  should be considered as the skein coordinate map of $\cS'( \bar \fS,\cP)$. Essentially, we work
with    $  \cS'( \bar \fS,\cP) $, instead of $  \cS\SM$, in the case of generalized marked surfaces.
\end{remark}

\subsection{Shear-to-skein map}  Let
$\oQ$ be the $\oD\times \oD$ submatrix of $Q_\D$. Recall the $\bQ$ is the $\oL\times \oL$ submatrix of $Q_\La$.
%Define the {\em collapsing map} $\omega: \oD\to \oL$, which is a contraction, by
%$\omega(a)= a$ if $a \in \oL$, and $\omega(d)=a$ if $d$ is an f-edge collapsing to $a$.
Let $\Omega\in\Mat(\oL \times \oD,\BZ)$ be the matrix defined by
 $\Omega(a,b)= 1$ if $\omega(b)=a$ and $\Omega(a,b)=0$ otherwise.
 %The relation between $\oQ$ and $\bQ$ is simple and is described by the following.
%In particular, $\omega(a,c_p)=0$ for all $p\in \ocP$.

\begin{lemma}
\lbl{r.100a}
One has
\be
\lbl{eq.55aa}
\bQ= \Omega \oQ \Omega^\dag.
\ee
\end{lemma}
\begin{proof}
Observe that
$$
\bQ(a,b)=\sum_{a'\in \omega^{-1}(a)} \sum_{b'\in \omega^{-1}(b)} \oQ(a',b'),
$$
which follows  easily from the explicit definition of $\oQ(a,b)$ and $\bQ(a,b)$. This is equivalent to
\eqref{eq.55aa}.
\end{proof}

Recall
$H$ is the $\oD\times \D$ submatrix of $Q_\D$.
Define
$\bH\in \Mat(\oL \times \D, \BZ)$ by
\be
\lbl{eq.55b}
\bH = \Omega H.
\ee
In other words, the $a$-row $\bH(a)$ of the matrix $\bH$ is given by
\be \lbl{eq.55g}
\bH(a)= \sum_{a'\in \omega^{-1}(a)} H(a')= \sum_{a'\in \omega^{-1}(a)} Q_\D(a').\ee

\begin{lemma}
\lbl{r.100}
(a) One has
\be
\bH P \bH^\dag = - 4 \bQ.
\ee

(b) The rank of $\bH$ is $|\oL|$.
\end{lemma}
\begin{proof}
(a)
Using \eqref{eq.55b}, \eqref{eq.55aa}, and then Lemma \ref{r.51a}, we have
 \begin{align*}
 \bH P \bH^\dag & = \Omega ( H P H^\dag) \Omega^\dag = -4 \Omega \oQ \Omega^\dag
 = - 4 \bQ.
 \end{align*}

 (b) Since $\rk(\Omega)= |\oL|$, the number of rows of $\Omega$, the left kernel of $\Omega$ is 0. Similarly,
 since $\rk(H)= |\oD|$ (by Lemma  \ref{r.51a}) the left kernel of $H$ is 0.  Hence the left kernel of
 $\bH= \Omega H$ is 0, which implies $\rk\bH= |\oL|$.
\end{proof}

\def\tiS{\tilde \cS}

% Besides, $\psi$ is injective.

\def\bvk{\bar {\vk}}
\def\bvkD{\bvk_\D}
\def\ZebD{\cY^\ev(\bD)}
\def\hZeD{\tilde \cY^\ev(\bD)}
\def\hZeDp{\tilde \cY^\ev(\bD')}
\def\hbXD{\tilde {\bar {\sX}} (\D)}
\def\hbXDp{\tilde {\bar {\sX}} (\D')}

 Recall that $\XhalfD = \BT(P,q^{1/4},x)$. From Lemma \ref{r.100} and Proposition \ref{r.homo}, we have the following.

  \begin{corollary}
  There exists a unique injective algebra
 homomorphism $\bPsi: \cY(\bD) \to \XhalfD$, such that for all $\bk\in \BZ^\oL$,
 \be
 \lbl{eq.bpsi}
   \bPsi(y^\bk) = x^{\bk\bH}
  \ee

  \end{corollary}

 \subsection{Existence of the quantum trace map}% We now formulate of the main  results of this section.
 The main result of \cite{BW0} is the construction of the quantum trace map
 $ \tr_q^\La: \ooS \to \ZeL.$
 We will show that $\tr_q^\La$ is the natural map $\bvpD: \ooS \to \bXD$, via the shear-to-skein map $\bar \psi$.

 \begin{proposition}\lbl{r.homo3}
  The map $\bvpD: \ooS \to \bXD$ is an %injective
 algebra homomorphism.

 \end{proposition}

 \begin{proposition}\lbl{r.homo4}
 One has $\bar \psi(\cY^\ev(\bD)) \subset \bXD$.
\end{proposition}

 \begin{theorem} Let $(\bsS,\cP)$ be a triangulated generalized marked surface,  with  triangulation $\La$, and with associate marked surface $\SM$. Assume that $\D$,  a triangulation of $\SM$, is a lift of $\La$.
 \lbl{r.main1}

 (a) In the diagram
\be
\lbl{eq.2head4}
\ooS \overset {\bvpD } \longrightarrow \bXD \overset \bPsi \rembed \ZeL
\ee
the image of $\bvpD $ is in the image of $\bPsi$, i.e. $\bvpD (\ooS) \subset \bPsi(\ZeL)$.

(b) Let $\bvkD: \ooS \to \ZebD$ be defined by $\bvkD= (\bPsi)^{-1} \circ \bvpD$. Then $\bvkD$ is equal to
the Bonahon-Wong quantum trace map $\tr_q^{\bD}$.

\end{theorem}

\def\cJ{\mathcal J}

Part (b) of the theorem implies that  $\bvkD$ depends only on $\bD$. That is, if $\D,\D'$ are two
triangulations of $\SM$ which are lifts of $\La$, then $\bvkD =
\bvk_{\D'}= \tr_q^{\La}$.

The remaining part of this section is devoted to  proofs of Propositions \ref{r.homo3}, \ref{r.homo4},
and Theorem~\ref{r.main1}. Note that Theorem \ref{thm.3} is a special case of Theorem \ref{r.main1}, when $\partial \bsS=\emptyset$.

\def\sC{\mathscr C}
\def\bCol{\overline{\Col}}
\def\bC{\overline{\sC}}
\def\Csao{\sC^*}

\subsection{Proof of Proposition \ref{r.homo3}}

Let $\sX_+(\D) \subset \sX(\D)$ be the $\cR$-submodule spanned by $X^\bk$ such that $\bk(c_p) \ge 0$ for all
$p\in \ocP$.   It is clear that $\sX_+(\D)$ is an $\cR$-subalgebra of $\sX(\D)$, and $\bXD \subset
\sX_+(\D)$.

\begin{lemma}
\lbl{r.embed}
 One has $\cS \subset \sX_+(\D)$.
\end{lemma}
\begin{proof}
 Suppose $\al\subset \Sigma$ is either a  $\cP$-knot or a $\cP$-arc.
 Define $\bk_\al\in \BZ^\D$ by $\bk_\al(a) =\mu(\al,a)$. Clearly if $a$ is a
 boundary edge, then $\bk_\al(a)=0$. Hence, $X^{-\bk_\al} \in \sX_+(\D)$.
 By \cite[Corollary 6.9]{Muller}, $X^{\bk_\al} \al  \in \sX_{++}(\D)\subset  \sX_{+}(\D)$.
 It follows that $\al \in X^{-\bk_\al}\sX_+(\D) \subset \sX_+(\D)$. Since $\cS$ is generated as an algebra
 by the set of all $\cP$-knots and $\cP$-arcs, we have $\cS \subset \sX_+(\D)$.
\end{proof}
\begin{proof}
[Proof of Proposition \ref{r.homo3}]
We will prove the stronger statement which says that  the restriction $\pi|_{\cS}: \cS \to \bXD$
 % $\bvpD: \ooS \to \bXD$
 is an $\cR$-algebra homomorphism.
Let $\cI$ be the two-sided ideal of $\sX_+(\D)$  generated by central elements $c_p, p \in \ocP$. Then
$\sX_+(\D) = \bXD \oplus \cI$, and the canonical projection
$
 \pi_+: \sX_+(\D) \to  \bX(\D)
$
is  the quotient map $\sX_+(\D)\to \sX_+(\D)/\cI = \bXD$.
It follows that  $\pi_+$  is an $\cR$-algebra homomorphism.
 Since $\pi|_{\cS}: \cS \to \bXD$ is the restriction of $\pi_+:\sX_+(\D)\to \bXD$
onto $\cS$, it is an $\cR$-algebra homomorphism.
\end{proof}

\def\bXHD{\bX^{(\frac 12)}(\D)}
\def\osC{\mathring{\sC}}
\def\bG{{ \G}}
\def\hbS{\hat {\sS}}
\def\hbD{{\hat {\bD}}}
\def\hbQ{{\hat {\bQ}}}
\def\bal{\bar \al}
\def\bs{{\bar s}}
\def\St{\operatorname{St}}
\def\hg{\hat g}
\def\hoD{\widehat{\oD}}
\def\haD{\hat D}
\def\hoQ{{\widehat{\oQ}}}

\subsection{$\D$-normal knots
 }
 \lbl{sec.simple}
 Let  $\al\subset \Sigma \setminus \pS$ be a   $\D$-normal  knot.
Recall that a state $s$ with respect to $\D$ is a map $s: \al\cap E_\D\to \{1,-1\}$, where $E_\D=\cup_{e\in \D}e$, see Section \ref{r.states}. By restricting $s$ to the subset $\al\cap E_\La$ we get a state $\bs$ of $\al$ with respect to $\La$. It may happen that $s$ is admissible but $\bs$ is not.
Let
 $\Col(\al,\D)$ be the set of all admissible states of $\al$ with respect to $\D$, and $\Col(\al,\bD)$ be the set of all admissible states of $\al$ with respect to $\bD$.
 For $s\in\Col(\al,\D)$ one has   $\bk_s \in \BZ^\oD$  defined by $\bk_s(a)= \sum_{u\in \al \cap a} s(u)$.
Similarly, for $r\in\Col(\al,\bD)$ one has    $\bk_r \in \BZ^\oL$  defined by $\bk_r(a)= \sum_{u\in \al \cap a} r(u)$.

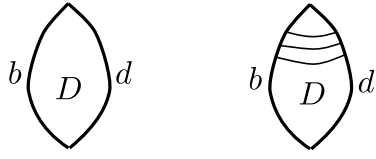
\FIGc{fake3a}{Left: $b$ and $d$ co-bound a disk $D$ (left). Right: the intersection  $D\cap \al$.}{2cm}

Suppose $b\neq d\in \oD$ such that $\omega (b)=\omega(d)\in \oL$. %(It may happen that $b=a$.)
 Then as $\cP$-arc in $\bsS$, $b$ and $d$ are $\cP$-isotopic, and hence co-bound a disk $D$ in $\bsS$, see Figure \ref{fig:fake3}. Since $\al$ is $\D$-normal, a connected component of $\al \cap D$ must have two end points with one in $b$ and one in $d$, see Figure \ref{fig:fake3a}.
 We say that a state $s\in\Col(\al,\D)$ is {\em $\omega$-equivariant on $b,d$} if $s(\gamma\cap b)=s(\gamma\cap d)$ for any connected component of $\al\cap D$. We say $s$ is {\em $\omega$-equivariant if} it is equivariant on any pair $b,d$ such that $\omega(b)=\omega(d)$.

Clearly if $s\in \Col(\al,\D)$ is $\omega$-equivariant, then $\bs$ is admissible.
Let $\Col^*(\al,\D)\subset \Col(\al,\D)$ be the subset of all $\omega$-equivariant states. The map
$s \to \bs$ is a bijection from $\Col^*(\al,\D)$ to $\Col(\al,\bD)$.

\begin{lemma} Suppose $s\in \Col^*(\al,\D)$. Then
\begin{align}
\lbl{eq.m6b}
\bk_s &= \bk_\bs \Omega
\no  {
%%%%%%%%%%%%%%%%
\\
\lbl{eq.m6a}
 \la \bk_s, \bk_{s'} \ra_\oQ &= \la \bk_\bs, \bk_{\bs'} \ra_\bQ.
 }
 %%%%%%%%%%%
\end{align}
\end{lemma}
\begin{proof} Let $a\in \oD$. From the  definition,
$
\bk_s (a) =\bk_{\bs}(\omega(a))
$, which is
which is equivalent to \eqref{eq.m6b}.
\end{proof}

% Recall that $H$ is the $\oD \times \D$ submatrix of $Q$.

Suppose $p\in \ocP$ and $\tau_p$ is the fake triangle having $c_p$ as an edge. Let the other two edges of $\tau_p$ be $e'_p, e''_p$ such that $c_p, e''_p, e'_p$ are counterclockwise, see Figure \ref{fig:fake3}. Then $\omega(e'_p)= \omega(e''_p)$.
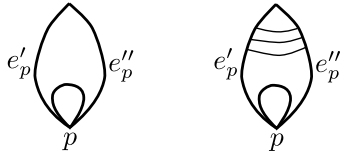
\FIGc{fake3}{Fake triangle $\tau_p$ (left) and its intersection with $\al$}{2cm}

\begin{lemma} \lbl{r.m1}

(a) Suppose $\bk \in \BZ^{\oD}$ then $\bk H (c_p)= \bk(e'_p) - \bk(e''_p)$.

(b) Suppose $s \in \St^*(\al,\D)$. Then $ x^{\bk_s H} \in \bXD$.

(c) Suppose $s\in \St(\al,\D)$. Then
\be
\lbl{eq.m9}
\pi(x^{\bk_s H}) = \begin{cases}
x^{\bk_\bs \bH} \quad & \text{if } s \in \St^*(\al,\D)\\
0  & \text{if } s \in \not \in \St^*(\al,\D).
\end{cases}
\ee

%(c) Suppose $s \in \St(\al,\D) \setminus\St^*(\al,\D)$. Then $\pi(x^{\bk_s H})=0$.
\end{lemma}

\begin{proof}

(a) is a special case of Lemma \ref{r.Qtau0}.
\def\tbk{\tilde {\bk}}

(b) Since $\bk_s$ is $\D$-balanced, $\bk_sH$ has even entries by Lemma \ref{r.even1}.
  Because $s\in \St^*(\al,\D)$, $\bk(e'_p) = \bk(e''_p)$ for all $p\in \ocP$. Part (a) shows that $(\bk_s H)(c_p)=0$ for all $p\in \ocP$, which means
$ x^{\bk_s H} \in \bXD$.

(c) The admissibility shows that $\bk_s(e'_p) \ge \bk(e''_p)$, and equality holds if and only if $s$ is $\omega$-equivariant on $e'_p, e''_p$.

First suppose $s\in \St^*(\al,\D)$. Part (b) and then \eqref{eq.m6b} show that
$$ \pi (x^{\bk_s H} )= x^{\bk_s H} = x^{\bk_\bs \Omega H} = x^{\bk_\bs \bH}.$$

Now suppose $s\not\in \St^*(\al,\D)$. Then there is a $p \in \ocP$ such that $s$ is not $\omega$-equivariant on $e'_p, e''_p$, and hence $\bk_s(e'_p) > \bk_s(e'')$. Part (a) shows that
 $(\bk_s H)(c_p)> 0$, and
 $(\bk_s H)(c_{p'})\ge 0$ for all other $p'\in \ocP$.  Hence, $\pi(x^{\bk_s H})=0$.
\end{proof}

\subsection{Proof of Proposition \ref{r.homo4}}

\begin{lemma}
\lbl{r.85}
One has $\bPsi(\cY^{(2)}(\bD))\subset \bX(\D)$.
\end{lemma}
\begin{proof} By definition, $\bPsi(y_a^{\pm 2})= x^{\pm 2 \bH(a)} = X^{\pm  \bH(a)} = X^{\pm \Omega H(a)}$. Note that every row of $\Omega$ is $\omega$-equivariant on $e'_p, e''_p$ for all $p\in \ocP$. Hence, by  Lemma \ref{r.m1}(a), $ \Omega H(a)(c_p)=0$, which means
$X^{\pm \Omega H(a)} \in \bXD.$ Since $y_a^{\pm 2}$ with $a\in \La$ generate $\cY^{(2)}(\bD)$, we have
$\bPsi(\cY^{(2)}(\bD))\subset \bX(\D)$.
\end{proof}
For a knot $\al \subset (\Sigma \setminus \pS) \subset \sS$ let $\bk_\al\in \BZ^\oL$ be defined by $\bk_\al(e)=\mu(\al,e)$ for all $e\in \oL$.
\begin{lemma}
\lbl{r.in2}
Suppose $\al$ is a $\bD$-simple knot in $\bsS$.
%Let $\bk_\al\in \BZ^\La$ be defined by $\bk_\al(a) = \mu(\al,a)$.
Then $\bPsi(y^{\bk_\al}) \in \bX(\D)$.
\end{lemma}
\begin{proof} We can assume that $\al$ is $\La$-normal. Then $\bk_\al \Omega= \bk_s\in \BZ^\oD$, where $s\in \St(\al,\D)$ is defined by $s (a\cap \al)= \mu(a,\al)= |a\cap \al|$. Then $s \in \St^*(\al,\D)$. By definition,
$$ \bPsi(y^{\bk_\al})= x^{\bk_\al \bH} = x^{\bk_\al \Omega H}=x^{ \bk_s H} \in \bXD,$$
where for the last inclusion we use  Proposition \ref{r.m1}(b).
\end{proof}

\begin{proof}
[Proof of Proposition \ref{r.homo4}] By Lemma \ref{r.gen.Ye}, $\cY^\ev(\bD)$ is generated by $\cY^{(2)}(\bD)$ and
$y^{\bk_\al}$ with $\bD$-simple knots $\al$. Lemmas \ref{r.85} and \ref{r.in2} show that
$\bPsi(\cY^\ev(\bD)) \subset \bXD$.
\end{proof}

\def\bare{\bar e}
\def\barl{\bar l}
\def\bff{\bar{\mathfrak f}}
\def\bgg{\bar{\mathfrak g}}
\def\ded{e}
\def\hLa{\hat \Lambda}
\def\hL{\hat \Lambda}
\def\bV{\bar V}

\subsection{Knots crossing an edge once}

 % Recall that $\D$ is obtained from $\La$ by adding $c_p, c\in \ocP$ and
 Suppose $\al\subset (\Sigma \setminus \pS) \subset \sS$ is a $\D$-normal oriented knot crossing an edge $a \in \oL$ once.
 For $r\in \Col(\al,\La)$ and $s\in \Col(\al,\D)$ one can define rational numbers $\uu(s)$ and $ \uu(r)$ as in Section\ref{sec.73}.
\begin{lemma} \lbl{r.m2}
Suppose $s\in \Col^*(\al,\D)$. Then $\uu(s)= \uu(\bs)$.
\end{lemma}
\begin{proof}
\def\htau{{\hat \tau}}
\def\hnu{{\hat \nu}}
\def\bpr{\overline{\pr}}

Recall that
\be
2 \uu(s) = -\sum_{\tau\in \cF(\D)} \ff(\tau;s), \quad
2 \uu(\bs) = -\sum_{\nu\in \cF(\La)} \ff(\nu;\bs).
\ee
Let $\cF^f(\D)$ be the set of all fake triangles of $\D$ and $\cF^*(\D)= \cF(\D) \setminus \cF^f(\D)$. The lemma clearly follows from the following two claims.

{\em Claim 1.} There is a bijection $\sigma: \cF(\La) \to \cF^*(\D) $ such that
$
 \ff(\nu;\bs) = \ff(\sigma(\nu), s).
$

\def\bA{\bar A}
{\em Claim 2.}
If  $\tau\in \cF^f(\D)$ then $\ff(\tau;s)=0$.

{\em Proof of Claim 1.} Suppose $\nu\in \cF(\La)$. Then $\nu$ contains one or several triangles of $\D$, and exactly one of them, denoted by $\sigma(\nu)$, is non-fake. The map $\sigma: \cF(\La) \to \cF^*(\D) $ is a bijection.

By splitting all the inner edges of $\La$, from $\nu$ one gets triangle $\hnu$, with a projection $\bpr: \hnu\to \nu$. Let $\cP_\nu$ be the set of vertices of $\hnu$.
Let $\tau=\sigma(\nu)$.
 One can identify $\htau$  with $(\bpr)^{-1}(\tau)\subset \hnu$. Then $\htau$ and $\hnu$ are $\cP_\nu$-isotopic, and $(\bpr)^{-1}(\al)$ intersects $\htau$ and $\hnu$ by the same patterns in the following sense:
First, there is a bijection $\sigma$ from the edges of $\hnu$ to the edges of $\htau$ such that
$a$ is $\cP_\nu$-isotopic to $\sigma(a)$. Second, if  $\beta_1, \dots, \beta_l$ are all the connected components of $(\bpr)^{-1}(\al)$ in $\hnu$, then all the connected components of $\pr^{-1}(\al)$ in $\htau$ are $\beta'_i=\beta_i\cap \htau$, $i=1,\dots,l$. See Figure \ref{fig:nonfake}.

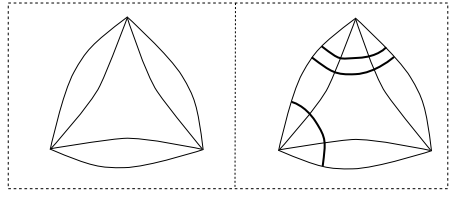
\FIGc{nonfake}{Left: outer triangle $\hnu$ is $\cP_\nu$-isotopic to inner triangle $\htau=\widehat{\sigma(\nu)}$.
Right: $(\bpr)^{-1}(\al)$ intersects $\hnu$ and $\htau$ by same pattern.}{2.5cm}

Let $z_i,t_i$ (resp. $z'_i,t'_i$) be the respectively the beginning point and the ending point of $\beta_i$ (resp. $\beta'_i$).
Then $V_\nu= \{ z_1,t_1,\dots,z_l, t_l\}$ and  $V_\tau= \{ z'_1,t'_1,\dots,z'_l, t'_l\}$.
%We can renumber the intervals $\beta_1,\dots,\beta_l$ such that the order of $V_\nu$ is
%$$ z_1 < t_1 < z_2 < t_2 < \dots < z_k < t_k.$$
The map $\sigma: V_\nu \to V_\tau$, defined by $\sigma(z_i)= z'_i$ and $\sigma(t_i)=t'_i$, preserves the order, and actually gives a bijection from the set $\{ u,v \in V_\nu, u \ll v\}$ to the set
$\{ u,v \in V_\tau, u \ll v\}$.

By definition~\eqref{eq.ff},
\begin{align}
 \ff(\tau;s)
\lbl{eq.55a} & = \sum_{ u,v \in V_\tau, u \ll v } Q_\htau({e(u)}, {e(v)}) s(u) s(v)\\
\lbl{eq.55c} \ff(\nu;r) &= \sum_{u,v \in V_\nu, u \ll v } Q_\hnu({e(u)}, {e(v)}) s(u) s(v).
\end{align}

The $\cP_\nu$-isotopy and the fact that $s\in \St^*(\D)$  show that for $u,v\in V_\nu$,
\be
  Q_\hnu(e(u),e(v))= Q_{\htau}(e(\sigma(u)), e(\sigma(v))), \quad \bs(u) = s(\sigma(u)).
\ee
 Hence, \eqref{eq.55a} and \eqref{eq.55c} show that $
 \ff(\tau;s) = \ff(\nu, \bs)
$, completing the proof of Claim 1.

{\em Proof of Claim 2.}
Suppose $\tau=\tau_p$ is a fake triangle, with edges $c_p, e''_p, e'_p$ in counterclockwise order, see Figure  \ref{fig:fake3}.  Suppose $\al\cap \tau$ consists of intervals whose lifts to $\htau$ are $\beta_1,\dots, \beta_l$. Let $u_i, v_i$ be the  end points of $\beta_i$ respectively on the lift of $e'_p$  and the lift of $e''_p$. By renumbering we can assume that each of $u_i, v_i$ is smaller than each of $u_j, v_j$ if $i<j$. By definition,
$$ \{ (u,v) \in (V_\tau)^2, \ u\ll v\}= \bigsqcup_{1\le i< j \le k} \{ (u_i,v_j), (v_i, u_j)\}.$$
Hence \eqref{eq.55a} can be rewritten as
\begin{align*}
\ff(\tau;s)  %& = \sum_{(u,v) \in (V_\tau)^2, \ u\ll v} Q_{\htau}(u, v) s(u) s(v)\\
&=\sum_{1\le i< j \le l} \left[ Q_{\htau}(e(u_i), e(v_j)) s(u_i) s(v_j) + Q_{\htau}(e(v_i), e(u_j)) s(v_i)s(u_j) \right]\\
&=\sum_{1\le i< j \le l} \left[ Q_{\htau}(e'_p,e''_p) s(u_i) s(v_j) + Q_{\htau}(e''_p, e'_p) s(v_i)s(u_j) \right]=0.
\end{align*}
 Here the last equalities holds since $s(u_i)= s(v_i)$ (because $s\in \St^*(\al,\D)$)and $Q_\htau$ is anti-symmetric. This completes the proof of Claim 2 and the lemma.
\end{proof}

\def\buu{\bar{\uu}}
\def\hZeLa{\tilde \cY^\ev(\bD)}
\def\fB{\mathfrak B}
\def\XoD{\sX_{++}(\oD)}
\def\bvphi{\bar\varphi}
\def\bPhi{\bar \Phi}

\begin{proposition}\lbl{r.68}
If $\al\subset \Sigma \setminus \pS$ is a $\D$-normal oriented knot crossing an edge $a\in \oL$ once, then
\be
\lbl{eq.m10}
\bar\varphi_\D(\al) = \sum_{r \in \St(\al,\bD)} q^{\uu(r)} x^{\bk_r \bH}.
\ee
\end{proposition}
\begin{proof}
By \eqref{eq.m1}, we have
$$ \al = \sum_{s \in \St(\al,\D)} q^{\uu(s)} x^{\bk_s H}.$$
Since $\pi(x^{\bk_s H}) =0$ unless $s \in \St^*(\al,\D)$ (by Lemma \ref{r.m1}(c)), we have
\begin{align*}
\bar\varphi_\D(\al)=\pi(\al) & = \sum_{s \in \St^*(\al,\D)} q^{\uu(s)} \pi (x^{\bk_s H}) \\
&= \sum_{s \in \St^*(\al,\D)} q^{\uu(s)} x^{\bk_{\bs} \bH} \quad \text{by Lemma \ref{r.m1}(b) }\\
&= \sum_{r \in \St(\al,\La)} q^{\uu(r)} x^{\bk_{r} \bH} \quad \text{by Lemma \ref{r.m2} }.
\end{align*}
For the last equality we also use the bijection $\St^*(\al,\D) \to \St(\al,\La)$ given by $s\to \bs$.
\end{proof}

\subsection{Proof of Theorem \ref{r.main1}}
\begin{proof}
[Proof of Theorem \ref{r.main1}]

(a)
Suppose $\al \subset \Sigma$ is a $\cP$-knot crossing an edge $a\in \oL$ once. Choose an  orientation of $\al$.
By Proposition \ref{r.68} and Equation \eqref{eq.bpsi},
 \be
 \lbl{eq.1s}
 \bar\varphi_\D(\al) = \sum_{r \in \St(\al,\bD)} q^{\uu(r)} x^{\bk_r \bH}
 =  \bPsi \left( \sum_{r \in \St(\al,\bD)} q^{\uu(r)} y^{\bk_r} \right),
 \ee
 which shows $\bar\varphi_\D(\al)$ is in the image of $\bPsi$. Since $\cP$-knots crossing one of the edges in  $\oL$ once generate $\ooS$, $\bar\varphi_\D(\ooS)$ is in the image of $\bPsi$. This proves part (a).

 (b) We have to show that $\bvk_\D:= (\bPsi)^{-1} \circ \bvpD$ coincides with $\tr_q^\La$.
 Identities  \eqref{eq.1s} and \eqref{eq.m1} show that $\bvk_\D =\tr_q^\La$ on a $\cP$-knot crossing an edge $a\in \oL$ once.  Since $\cP$-knots crossing one of the edges in  $\oL$ once generate $\ooS$,
  we have $\bvk_\D =\tr_q^\La$.
\end{proof}

\begin{remark}
In \cite{BW0}, the quantum trace is constructed with a domain bigger than $\ooS$. Namely, the domain is so called {\em state skein algebra} of a generalized marked surface. Using \eqref{eq.arctoskein} one can also recover the quantum trace map in this bigger domain through the skein coordinates.
\end{remark}

\def\tYbl{\tilde \cY^\ev}

\appendix

\def\Stwo{\cS^{(2)}}
\def\tStwo{\ttS^{(2)}}
\def\cQ{\mathcal Q}

\section{Proof of Theorem \ref{r.shearchange}}
\lbl{sec.shearchange} Suppose $\D$ is a triangulation of the marked surface $\SM$. We identify $\ooS$ with a subset of $\XD$ via $\varphi_\D$. Thus,
 $X_a=a$ . We also write $a^{1/2} = x_a \in \XhalfD$ for $a\in \D$,
and  use the notation $Y_a = y_a^2 \in \cY^{(2)}(\D)\subset \cY(\D)$.

\subsection{The case of $\cY^{(2)}(\D)$}
Suppose $\cQ$ is a $\cP$-quadrilateral (see Section \ref{r.42}), with edges $a,b,c,d$ in some counterclockwise order. Define $[\cQ]:= \{ [a b^{-1} c d^{-1}], [a b^{-1} c d^{-1}]^{-1}\}$, which does not depend on the counterclockwise order.
 For now define $\cS^{(2)}$ to be the $\cR$-subalgebra of $\cS$ generated by all $[\cQ]$, where $\cQ$ runs the set of all $\cP$-quadrilaterals. Let $\ttS^{(2)}$ be the skew field of $\cS^{(2)}$, i.e. the set of all elements of the form $\al \beta^{-1}\in \ttS$ with $\al, \beta\in \Stwo, \beta \neq 0$.

Suppose $a\in \oD$, where $\D$ is a triangulation of $\SM$. Let $\cQ_{a,\D}$ be the $\cP$-quadrilateral consisting of the two triangles having $a$ as an edge. By \eqref{eq.shear},
\be
\{\psi(Y_a), \psi(Y_a)^{-1}\}= [\cQ_{a,\D}].
\ee
Since $\{ Y_a^{\pm1} \mid a \in \oD\}$ generates $\tY^{(2)}(\D)$, we have
\be
\lbl{eq.0s}
\tpsi_\D(\tY^{(2)}(\D)) \subset \tStwo.
\ee

\begin{lemma}\lbl{r.A1}
Suppose
 $\D'$ is the flip of $\D$ at $a\in \oD$ then $\tpsi_{\D'}(\tY^{(2)}(\D')\subset \tpsi_\D(\tY^{(2)}(\D)$.
\end{lemma}

 \begin{proof} Suppose the  boundary edges of $\cQ_{a,\D}$ are denoted as in Figure \ref{fig:mutation}.
 We have $\D'= \D \setminus \{a\} \cup \{a*\}$.
 Since $\tY^{(2)}(\D')$ is generated by $Y_v,v\in \oD'$, it is enough to show that $\psi_{\D'}(Y_v) \in \tY^{(2)}(\D)$. It is clear that if $v\not\in\{a,b,c,d,e\}$, then $\psi_{\D'}(Y_v) = \psi_{\D}(Y_v) \in \psi_{\D}(\tY^{(2)}(\D))$. Besides,
 \be
 \lbl{eq.s0}
 \psi_{\D'}(Y_{a^*}) = \psi_{\D}(Y_{a})^{-1} \in \psi_{\D}(\tY^{(2)}(\D)).
 \ee

   It remains to show $\psi_{\D'}(Y_v) \in \tY^{(2)}(\D)$ for $v\in \{b,c,d,e\}$. Since there is no self-folded triangle, if four edges $\{b,c,d,e\}$ are not pairwise distinct, then there  is exactly one pair of two opposite edges which are the same. Thus we have 3 cases (A), (B), and (C) below.

(A) Four edges $b,c,d,e$ are pairwise distinct. Using the explicit formula \eqref{eq.shear}, we have
\begin{align}
\psi_{\D'}(Y_{v }) &= \psi_{\D}(Y_{v} + [Y_v Y_a]) \quad \text{for $v\in \{b, d\}$} \lbl{eq.s1}\\
\psi_{\D'}(Y_{v }^{-1}) &= \psi_{\D}(Y_{v}^{-1} + [Y_v^{-1} Y_a^{-1}]) \quad \text{for $v\in \{c, e\}$}. \lbl{eq.s2}
\end{align}
(B) $b=d$, otherwise $b,c,d,e$ are pairwise distinct. Using  \eqref{eq.shear}, we have
\begin{align}
\psi_{\D'}(Y_{b }) &= \psi_{\D}(Y_{b} + (q^{1/2}+ q^{-1/2})[Y_aY_b] + [Y_a^2 Y_b])  \lbl{eq.s3} \\
\psi_{\D'}(Y_{v }^{-1}) &= \psi_{\D}(Y_{v}^{-1} + [Y_v^{-1} Y_a^{-1}]) \quad \text{for $v\in \{c, e\}$}. \lbl{eq.s4}
\end{align}
(C) $c=e$, otherwise $b,c,d,e$ are pairwise distinct. Using  \eqref{eq.shear}, we have
\begin{align}
\psi_{\D'}(Y_{c }^{-1}) &= \psi_{\D}(Y_{c}^{-1} + (q^{1/2}+ q^{-1/2})[Y_a^{-1}Y_c^{-1}] + [Y_a^{-2} Y_c^{-1}])   \lbl{eq.s5}\\
\psi_{\D'}(Y_{v }) &= \psi_{\D}(Y_{v} + [Y_v Y_a]) \quad \text{for $v\in \{b, d\}$}.  \lbl{eq.s6}
\end{align}
In each case,  $\psi_{\D'}(Y_v) \in \psi_\D(\in \tY^{(2)}(\D))$ for $v\in \{b,c,d,e\}$.
 \end{proof}

\begin{proposition}
\lbl{r.Ytwo}
One has $\tpsi_\D(\tY^{(2)}(\D))=\tStwo  $, which does not depend on $\D$.
\end{proposition}
\begin{proof}Lemma \ref{r.A1}, with $\D,\D'$ exchanged, shows that $\tpsi_{\D'}(\tY^{(2)}(\D'))= \tpsi_\D(\tY^{(2)}(\D))$. Any two triangulations are related by sequence of flips. Hence $\tpsi_\D(\tY^{(2)}(\D))$ does not depend on the triangulation $\D$.

Suppose $\cQ$ is a $\cP$-quadrilateral. Let $a$ be a diagonal of $\cQ$. Then the collection of $a$ and the edges of $\cQ$ can be extended to a triangulation $\D$ of $\SM$. Thus,  $[\cQ]= \psi_\D(a)^{\pm 1} \in \tpsi_\D(\tY^{(2)}(\D))$. Since all the $[\cQ]$ generate $\tStwo$, we have $\tStwo \subset \tpsi_\D(\tY^{(2)}(\D))$. Together with \eqref{eq.0s}, we have $\tStwo = \tpsi_\D(\tY^{(2)}(\D))$.
 \end{proof}

\begin{lemma}\lbl{r.a3}
For any two triangulations $\D,\D'$, the algebra isomorphism $\Theta_{\D,\D'}: \tY^{(2)}(\D')\to \tY^{(2)}(\D)$ defined by $\Theta_{\D,\D'} = \tpsi_\D^{-1} \circ \tpsi_{\D'}$ coincides with the coordinate change map $\Phi_{\D\D'}$ in \cite{Liu}\footnote{Our $q$ is equal to $q^2$ of \cite{Liu}}.
\end{lemma}
\begin{proof}
It is enough to consider the case when $\D'$ is obtained from $\D$ by the flip at $a\in \oD$, with notations of edges $b,c,d,e$ as in Figure \ref{fig:mutation}. From Identities \eqref{eq.s0}--\eqref{eq.s6} and cases (A), (B), (C) as in the proof of Proposition \ref{r.Ytwo}, we have
\begin{align*}
\Theta_{\D,\D'} (Y_v) &= Y_v \quad \text{if  $v\not \in \{a,b,c,d,e\}$} \\
\Theta_{\D,\D'} (Y_{a^*}) &= (Y_a)^{-1} \\
\Theta_{\D,\D'} (Y_v) &= Y_{v} + [Y_v Y_a]  \quad \text{in case (A) and $v \in \{b,d\}$}\\
\Theta_{\D,\D'} (Y_v^{-1}) &= Y_{v}^{-1} + [Y_v^{-1} Y_a^{-1}]  \quad \text{in case (A) and $v \in \{c,e\}$}\\
\Theta_{\D,\D'} (Y_b) &= Y_{b} + (q^{1/2}+ q^{-1/2})[Y_aY_b] + [Y_a^2 Y_b]  \quad \text{in case (B)}\\
\Theta_{\D,\D'} (Y_v^{-1}) &= Y_{v}^{-1} + [Y_v^{-1} Y_a^{-1}]  \quad \text{in case (B) and $v \in \{c,e\}$}\\
\Theta_{\D,\D'} (Y_c^{-1}) &= Y_{c}^{-1} + (q^{1/2}+ q^{-1/2})[Y_a^{-1}Y_c^{-1}] + [Y_a^{-2} Y_c^{-1}]  \quad \text{in case (C)}\\
\Theta_{\D,\D'} (Y_v) &= Y_{v} + [Y_v Y_a]  \quad \text{in case (C) and $v \in \{b,d\}$}\\
\end{align*}
Comparing with the formulas in \cite{Liu}, we see that our $\Theta_{\D,\D'}$ is equal to $\Phi_{\D,\D'}$ in \cite{Liu}.
\end{proof}

\def\Tub{\mathrm{Tub}}

\subsection{Image of $\D$-simple curve under $\psi$}
 Suppose $\al$ is a $\D$-normal knot.
 Recall that $\al$ is $\D$-simple if $\mu(\al,a)=|\al\cap a|\le 1$ for all $a\in \cE(\al,\D)$. We say $\al$ is {\em almost $\D$-simple} if $\mu(\al,a)\le1$ for all $a\in \cE(\al,\D)$ except for an edge $d$, where $\mu(\al,d)=2$.
If $|\al\cap e|=1$ then $\al$ passes $e$ by one of the four patterns described in Figure \ref{fig:square6}. Let $\bk_{\al,\D}\in \BZ^\D$ be defined by $\bk_{\al,\D}(a)= |\al \cap a|$.
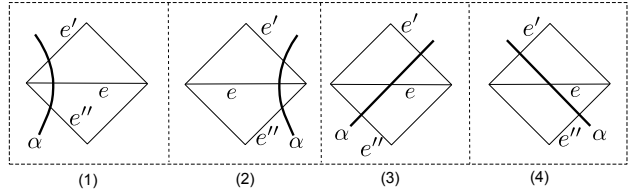
\FIGc{square6}{Patterns: (1) \& (2): unchanged pattern (3) left-right (4) right-left}{2.5cm}

\def\bve{\boldsymbol{\ve}}
\def\bod{\boldsymbol{\delta}}

\begin{lemma}\lbl{r.psi3}
 Suppose $\al$ is $\D$-simple or almost $\D$-simple knot. Then
 \begin{align}
  \lbl{eq.psi3}
\psi_\D(y^{\bk_{\al,\D}}) &=  X^{\bve_\al}=\left [ \prod_{e\in \cE(\al,\D)}   e^{\bve_\al(e)  }\right]
 \end{align}
 where $\bve_\al\in \BZ^\D$ is defined by $\bve_\al(e)=0$ if $|\al \cap e |>1$ or $\al$ passes $e$ in the unchanged pattern,
 $\bve_\al(e)=1$ if $\al$ passes $e$ in the right-left pattern, and $\bve_\al(e)=-1$ if $\al$ passes $e$ in the left-right pattern.
\end{lemma}
\begin{proof} Denote $\bk=\bk_{\al,\D}$.
 For $a\in \D$ let $\bod_a \in \BZ^{\D}$ be defined by $\bod_a(e)= 1$ if $a=e$ and $\bod_a(e)=0$ otherwise.
 Note that $\bod_e Q_\tau\neq 0$ only if $e$ is an edge of $\tau$.

Suppose $a,c$ are edges of a triangle $\tau\in \cF(\D)$.  From the explicit formula \eqref{eq.Qtau} of $Q_\tau$, we have
\be
\lbl{eq.Qtau5}
(\bod_a + \bod_c) Q_\tau = Q_\tau(a,c) \bod_a + Q_\tau(c,a) \bod_c= Q(a,c) \bod_a + Q(c,a) \bod_c,
\ee
where the last equality holds since in a marked surface, a pair $a,c$ cannot be edges of two different triangles of $\D$.

First suppose $\al$ is $\D$-simple. Then $\bk = \sum_{a\in \cE(\al,\D)} \bod_a$. For $\tau\in \cF(\al,\D)$ let $a_\tau, c_\tau$ be its edges that intersect $\al$.
Using \eqref{eq.Qtau5}, one has
\be
\lbl{eq.84}
\bk Q_\tau =  (\bod_{a_\tau} + \bod_{c_\tau}) Q_\tau = Q(a_\tau,c_\tau) \bod_{a_\tau} + Q(c_\tau,a_\tau) \bod_{c_
 \tau}.
\ee
Hence,
\begin{align}
\bk Q & = \sum_{\tau} \bk Q_\tau = \sum_\tau Q(a_\tau,c_\tau) \bod_{a_\tau} + Q(c_\tau,a_\tau) \bod_{c_
 \tau}  \notag \\
 & = \sum_{e \in \cE(\al,\D)} (Q(e',e) + Q(e'',e)) \bod_e  \lbl{eq.Qtau1d}.
\end{align}
where $e',e''$ are edges in $\cE(\al,\D)$ neighboring to $e$, i.e. $e',e''$ are edges of triangles having $e$ as an edge, see Figure \ref{fig:square6}. By inspecting all  cases, we can check that $(Q(e',e) + Q(e'',e))/2 = \bve_\al(e)$. Hence, from \eqref{eq.Qtau1d} we get
$$ \psi_\D(y^{\bk})=  x^{\bk Q} = \left [ \prod_{e}   e^{\bve_\al(e)  }\right].$$

 Suppose now $\al$ is almost $\D$-simple. Let $d\in \cE(\al,\D)$ be the edge with $|\al \cap d|=2$, and $\tau_1, \tau_2$ be the triangles having $d$ as an edge. Then $\al$ must intersect $\tau_1 \cup \tau_2$ as in Figure \ref{fig:square7}. We also use the notations for edges neighboring to $d$ as in Figure \ref{fig:square7}, with $\tau_1$ having $d,c,v$ as edges.
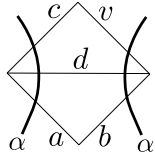
 \FIGc{square7}{Edge $d$ with $|\al\cap d|=2$ and its neighboring edges $a,b,c,v$}{2cm}

 For all $\tau\in \cF(\al,\D)$ other than $\tau_1, \tau_2$, we still have \eqref{eq.84}. Using \eqref{eq.Qtau5}, we can calculate
 \begin{align*}
 \bk(Q_{\tau_1} + Q_{\tau_2})&= (\bod_a+\bod_b + \bod_c \bod_v +2\bod_d)(Q_{\tau_1} + Q_{\tau_2})\\
 &= (\bod_a+\bod_d) Q_{\tau_1} (\bod_v+\bod_d) Q_{\tau_1}   (\bod_b+\bod_d) Q_{\tau_2}   (\bod_c+\bod_d) Q_{\tau_2} \\
 &= Q(d,v) \bod_v + Q(d,a)  \bod_a+ Q(d,b) \bod_b + Q(d,c) \bod_c.
 \end{align*}
Note that there is no $\bod_d$ in the final expression. From here one get
$$ \bk Q= \sum_{e\in \cE(\al,\D), e \neq d} \bve_\al(e) \bod_e.$$
Using $\psi_\D(y^{\bk})=  x^{\bk Q}$, we get \eqref{eq.psi3} for almost $\D$-simple knots.
\end{proof}

\subsection{The case of $\tY^\ev(\D)$} %It is convenient to extend $\tpsi_\D$ to an algebra homomorphsim from the skew field of $\YeD$ to $\ttS$.
\begin{lemma}\lbl{r.tpsi} Suppose $\D'$ is obtained from $\D$ by the flip at $a\in \oD$, and
 $\al$ is  $\D'$-simple knot. Then
 \be
\lbl{eq.psi0} \tpsi_{\D'} (y^{\bk_{\al,\D'}})= \tpsi_{\D} (y^{\bk_{\al,\D}})
 \ee
 except when $a^*\in \cE(\al,\D')$ and $\al$ passes $a^*$ by the  left-right or  right-left pattern. One has
\begin{align}
\lbl{eq.psi6} \tpsi_{\D'} (y^{\bk_{\al,\D'}})&= \tpsi_{\D} (y^{\bk_{\al,\D}} + [Y_a^{-1} y^{\bk_{\al,\D}}]) \quad
 \text{if $\al$ passes $a^*$ by right-left pattern}\\
\lbl{eq.psi7} \tpsi_{\D'} (y^{-\bk_{\al,\D'}})&= \tpsi_{\D} (y^{-\bk_{\al,\D}} + [Y_a y^{-\bk_{\al,\D}}]) \quad \text{if $\al$ passes $a^*$ by left-right pattern }.
\end{align}

\end{lemma}
\begin{proof}  After an isotopy we can assume that $\al$ is $\D'$-normal.
There are 2 cases: $a^* \in  \cE(\al,\D')$ and $a^* \not\in  \cE(\al,\D')$.

(i) Case $a^* \not \in \cE(\al,\D')$. Subcase (ia) $a\not \in \cE(\al,\D)$. Then $\cE(\al,\D)= \cE(\al,\D')$, and $\tpsi_{\D'} (y^{\bk_{\al,\D'}})= \tpsi_{\D} (y^{\bk_{\al,\D}})$ since both are equal to the right hand side of
 \eqref{eq.psi3}.

 Subcase (ib) $a\in \cE(\al,\D)$. Then $\al$ intersects $S$ like in Figure \ref{fig:square2}(a), where $\al$ is $\D$-simple,  or Figure \ref{fig:square2}(b), where $\al$ is almost $\D$-simple.
 In each case, we have \eqref{eq.psi0} due to Lemma \ref{r.psi3}.

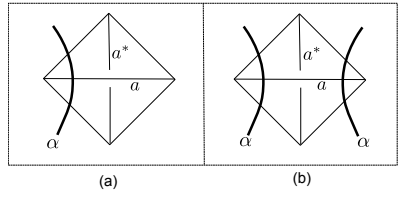
 \FIGc{square2}{Intersection of $\al$ with $\cQ_{\al,\D'}$}{2.5cm}

(ii) $a^* \in  \cE(\al,\D')$. Then $\al$ intersects $\cQ_{\al,\D'}$ in one of the four patterns described in Figure \ref{fig:square6}. In the first two cases, identity \eqref{eq.psi0} is proved already in subcase (ib) above, by switching $a\leftrightarrow a^*$.

Suppose $\al$ passes $a^*$ in the right-left pattern, with edge notations as in Figure \ref{fig:square5}.
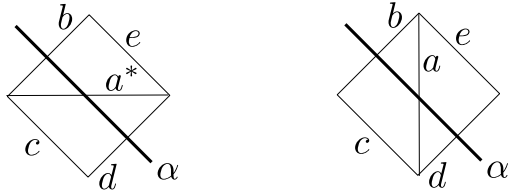
\FIGc{square5}{Left: $\al$ passes $a^*$ by right-left pattern. Right: the flip.}{2.5cm}

Denote $Q=Q_\D,Q'=Q_{\D'}$, $\bk=\bk_{\al,\D}$, and $ \bk'=\bk_{\al,\D'}$. Let $S= \cQ_{a,\D}=\cQ_{a^*,\D'}$ which is the support of the flip and is bounded by the 4 edges $b,c,d,e$. One might have $b=d$ or $c=d$.
Let $\cF_1$ (resp. $\cF'_1$) be the two triangles of $\D$ (resp. $\D'$) in $S$, and $\cF_2= \cF(\al,\D) \setminus \cF_1$, $\cF_2'= \cF(\al,\D') \setminus \cF'_1$.
Define
$$Q_1= \sum_{\tau \in \cF_1} Q_ \tau, \quad Q_1'= \sum_{\tau \in \cF'_1} Q_ \tau, \quad Q_2= \sum_{\tau \in \cF_2} Q_ \tau, \quad  Q'_2= \sum_{\tau \in \cF'_2} Q_ \tau.$$
Using \eqref{eq.Qtau5}, we get
$$ \bk' Q'_1= 2 \bod_{a^*} -\bod_b -\bod_d, \quad  \bk Q_1= -2 \bod_{a^*} +\bod_b +\bod_d,$$
which, together with $Q= Q_1 + Q_2$ and $Q'= Q'_1+ Q_2$, gives
\begin{align}
\lbl{eq.psi4}\tpsi_{\D'} (y^{\bk'})= x^{\bk' Q'}= x^{\bk' Q'_1 + \bk' Q'_2 } & = \left [ a^* b^{-1/2} d^{-1/2} x ^{\bk' Q'_2} \right ] \\
\lbl{eq.psi5}\tpsi_{\D} (y^{\bk})& = \left [ a^{-1} b^{1/2} d^{1/2} x ^{\bk Q_2}  \right ].
\end{align}
Since $\bk(e)= \bk'(e)$ for $e\not \in \{a, a^*\}$, we have $x ^{\bk Q_2} = x ^{\bk' Q'_2}$ as elements in $\XhalfD$.
Using $a^* = [bda^{-1}] + [ce a^{-1}]$ (see \eqref{eq.aa1}) in \eqref{eq.psi4} and a simple commutation calculation, we have
$$ \tpsi_{\D'} (y^{\bk_{\al,\D'}})= \left [ a^{-1} b^{1/2} d^{1/2} x ^{\bk Q_2} \right ] + \left [ (b^{-1} c d^{-1} e) (a^{-1} b^{1/2} d^{1/2} x ^{\bk Q_2} )  \right ] = \tpsi_{\D} (y^{\bk_{\al,\D}}) + \tpsi_{\D} (Y^{-1}_ay^{\bk_{\al,\D}}),$$
where the last equality follows from \eqref{eq.psi5}. This proves  \eqref{eq.psi6}. The other \eqref{eq.psi7} is proved similarly.
This completes the proof of the lemma.
\end{proof}

\begin{proof}[Proof of Theorem \ref{r.shearchange}]
(a) Lemma \ref{r.tpsi}  and Proposition \ref{r.Ytwo} show that if $\D'$ is obtained by a flip at $a\in \oD$, then $\tpsi_{\D'} (\tY^\ev(\D')) \subset \tpsi_{\D} (\tY^\ev(\D))$.
 Switching $\D \leftrightarrow \D'$ we get the reverse inclusion, and hence $\tpsi_{\D'} (\tY^\ev(\D')) = \tpsi_{\D} (\tY^\ev(\D))$. Since any two triangulations are related by a sequence of flips, we have $\tpsi_{\D'} (\tY^\ev(\D')) = \tpsi_{\D} (\tY^\ev(\D))$ for any two triangulations $\D,\D'$. The fact that $\tpsi_{\D'} (\tY^{(2)}(\D')) = \tpsi_{\D} (\tY^{(2)}(\D))$ was proved in Proposition \ref{r.Ytwo}. This proves part (a).

(b) Again the statement is reduced to the case when $\D'$ is obtained by a flip at $a\in \oD$.
The fact that $\Theta_{\D\D'}$ on $\tY^{(2)}(\D')$ coincides with the coordinate change map of \cite{Liu} was proved in Lemma~\ref{r.a3}.
Suppose $\al$ is a $\D'$-simple knot. From Lemma \ref{r.tpsi}, we have
$$ \Theta_{\D\D'} (y^{\bk_{\al,\D'}})= y^{\bk_{\al,\D}}$$
unless when $\al$ passes $a^*$ in the right-left or left-right patterns, and in those cases
\begin{align*}
\Theta_{\D\D'} (y^{\bk_{\al,\D'}}) & = y^{\bk_{\al,\D'}} + y_a^{-2} y^{\bk_{\al,\D}} \quad \text{if $\al$ passes $a^*$ in right-left pattern}\\
\Theta_{\D\D'} \left (y^{-\bk_{\al,\D'}}\right) & = y^{-\bk_{\al,\D'}} + y_a^{2}\, y^{-\bk_{\al,\D}} \quad \text{if $\al$ passes $a^*$ in left-right pattern}.
\end{align*}
Comparing with the formulas in \cite{Hiatt}, we see that our $\Theta_{\D\D'}$ and $\Phi_{\D\D'}$ of \cite{Hiatt} agree on $y^{\bk_{\al,\D'}}$. Since $\tY^\ev(\D')$ is generated by $\tY^{(2)}(\D')$ and $y^{\bk_{\al,\D'}}$, we conclude that our $\Theta_{\D\D'}$ and $\Phi_{\D\D'}$ of \cite{Hiatt} coincide.

(c) is Proposition \ref{r.Ytwo}. This completes the proof of Theorem \ref{r.shearchange}.
\end{proof}

\def\ot{\otimes}
\def\d{\delta}

%%%%%%%%%%%%%%%%%%%%%%%%%%%%%%%
%%%%%%%%%%%%%%%%%%%%%%%%%%%%%%%

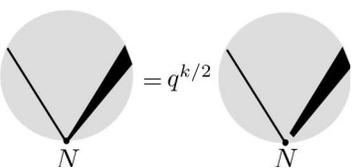
\end{document}